\documentclass[12pt]{amsart}
\usepackage[margin=1in,letterpaper]{geometry}
\usepackage{graphicx,amssymb,amsmath,amsthm,enumitem,xcolor} 
\usepackage[scr=boondox]{mathalpha}
\usepackage[colorlinks, linktocpage, breaklinks]{hyperref}

\newtheorem{theorem}{Theorem}[section]	
\newtheorem{lemma}[theorem]{Lemma}
\newtheorem{proposition}[theorem]{Proposition}
\newtheorem{corollary}[theorem]{Corollary}
\newtheorem{definition}[theorem]{Definition}
\newtheorem{example}[theorem]{Example}
\newtheorem{question}{Question}

\usepackage{tikz}

\newcommand{\dualarrows}{{\,\begin{tikzpicture}[baseline={([yshift=-0.7ex]current bounding box.center)}]
        \draw[line width=0.7pt, ->] (0,0.27) -- (0.4,0); \draw[line width = 0.7pt,->] (0,0) -- (0.4,0.27);
    \end{tikzpicture}\,}}

\title{On Large Covers and the Strong Closed Discrete Game}

\author{Christopher Caruvana}
\address{School of Sciences, Indiana University Kokomo, 2300 S. Washington Street, Kokomo, IN 46902 USA}
\email{caruvana@gmail.com}
\urladdr{https://ccaruvana.github.io}

\author{Jared Holshouser}
\address{Unaffiliated}
\email{JHolshouser1321@gmail.com}
\urladdr{https://jaredholshouser.github.io/}

\date{\today}



\begin{document}

\begin{abstract}
    We strengthen a recent selection game equivalence proved by L. Chiozini
    involving spaces and their associated spaces of continuous real-valued functions.
    This strengthening establishes that the Rothberger game is perfect-
    and Markov-information dual to the strong closed discrete game
    on the space of real-valued continuous functions.
    In the process, we extend the theory of strategic equivalences for the second player
    in variations of the Menger and Rothberger games involving large covers
    while noting explicitly how large covers present serious obstacles
    to applying strategy translation results used by the authors in similar
    scenarios in previous papers.
    We also briefly comment on subbasic selection games introduced by D.~Guerrero S\'{a}nchez
    and V.V.~Tkachuk.
\end{abstract}

\maketitle

\section{Introduction}

In a recent paper \cite{Chiozini}, L.~Chiozini asserted that the point-open game on a Tychonoff space \(X\)
is perfect-information equivalent to a modified closed discrete selection game on \(C_p(X)\),
the strong closed discrete game.
We take this result and expand it to the context of limited-information strategies.
In the process of proving the expanded result, we discovered that a crucial component, Lemma 3.4(a), of
Chiozini's argument is not true as stated, but a slight modification, which is surely in the spirit of the lemma,
guarantees its conclusion.
In the context of limited-information strategies, the modification is not slight and the details of our proofs
rely heavily on the combinatorics of large covers.
This paper then has two main goals:
to extend the results of Chiozini and to clarify the combinatorics of large covers.

Our primary contribution to the theory of large covers is the establishment of many results
concerning the second player's strategies in games like
\(\mathsf G_{\mathrm{fin}}(\Lambda_X, \Lambda_X)\) and \(\mathsf G_1(\Lambda_X, \Lambda_X)\).
It is in this study that we discover an important difference between the known theory
of strategic implications between these types of games and those where the cover types
are ideal-covers.
In particular, if we restrict our attention to \(T_1\) spaces,
\cite[Theorem 4.17]{CCHMengerRothbergerSurvey} establishes that the properties
of being Markov Rothberger and Markov \(\omega\)-Rothberger are equivalent to the
space being countable; by way of contrast, Theorem
\ref{thm:CountableNotMarkoveLambdaRoth} shows that the property of being Markov
\(\lambda\)-Rothberger fails to hold even for the countable discrete space.

The paper is generally organized as follows.
We finish this introduction by elaborating on the paper's motivation
in Section \ref{subsec:Motivation};
Section \ref{sec:Definitions} establishes all definitions and notation to be used
with some additional commentary on when certain cover types exist;
Section \ref{section:CoveringPrinciples} offers brief commentary on subbasic
selection games;
Section \ref{section:LargeCovers} offers a thorough treatment of large covers
in the context of standard selection games;
and Section \ref{section:TheMainTheorem} is devoted to the proof of the main theorem.

\subsection{Motivating Commentary} \label{subsec:Motivation}

We offer some commentary here regarding the motivation for this paper
and refer the reader to Section \ref{sec:Definitions} for explicit
definitions.

We first note with Example \ref{example:FinitelyManyFunctions}
that Lemma 3.4(a) of \cite{Chiozini} is false as stated.
However, as we will clarify in the proof of Lemma \ref{lem:sCDBelow},
the lemma is correct in spirit; that is, to guarantee the conclusion of the lemma,
one need only guarantee that the sequence of functions chosen are pairwise distinct.
This is not a problem in the context of perfect-information
strategies, which is the context of \cite{Chiozini}.
However, when considering limited-information strategies, it does
become a legitimate obstacle, which we will demonstrate how to overcome.

The origin of this issue seems to be in the difference between treating
the result of the second player's choices as a sequence versus as a set;
compare, for example, Definition 2.1 and Definition 3.1 of \cite{Chiozini}.
Particularly, in Definition 2.1, an open cover \(\mathscr U\) of a space \(X\)
is said to be a large cover if \(\{ U \in \mathscr U : x \in U \}\) is infinite
for each \(x \in X\).
In the context of the game \(\Lambda FO(X)\), the condition determining the victory
of the first player is in terms of the indices, not the open sets themselves.
Note that, for a countable space \(X\), by passing through a bijection \(\omega \to \omega^2\),
the first player can guarantee that the second player plays a set \(\{ U_n : n \in \omega \}\)
of open subsets such that, for each \(x \in X\), \(\{ n \in \omega : x \in U_n \}\) is an infinite
subset of \(\omega\).
However, it need not be the case that \(\{ U_n : x \in U_n \}\) is infinite.
Indeed, with Theorem \ref{thm:CountableNotMarkoveLambdaRoth} we will show that,
if the second player forgets their previous moves, in most spaces of interest,
they cannot guarantee a legitimate large cover of the space.

\begin{example} \label{example:FinitelyManyFunctions}
    Any finite set of functions is strongly closed discrete, but a finite set of functions
    which respects Process A of \cite{Chiozini} may generate a large cover.
\end{example}
\begin{proof}
    Obviously, if \(g_n = \mathbf 0\), the constant zero function, then Process A is
    respected but the resulting set of open subsets of \(X\) may be a large cover of \(X\)
    when \(X\) is infinite.
    We show that this can still occur if we require the functions to be nonzero.

    Our space will be \(\omega\) with the usual discrete topology.
    Let \(\gamma : \wp(\omega) \setminus \{ \varnothing \} \to \omega\) be a choice
    function such that, if \(A \subseteq \omega\) is such that \(A \setminus \{0,1\} \neq \varnothing\),
    then \(\gamma(A) \not\in \{ 0, 1 \}\).
    For each \(n \in \omega\), let \(W_{2n} = \omega \setminus \{0\}\) and
    \(W_{2n+1} = \omega \setminus \{1\}\).
    Then define \(F_n\) and \(p_n \in \omega\) recursively as follows.
    Initialize \(F_0 = \varnothing\) and \(p_0 = \gamma(W_0 \setminus F_0) = \gamma(W_0) \not\in \{0,1\}\).
    For \(n \in \omega\), let \(F_{n+1} = F_n \cup \{ p_n \}\) and
    \(p_{n+1} = \gamma(W_{n+1} \setminus F_{n+1})\).
    By construction \(\{ p_n : n \in \omega\}\) consists of pairwise distinct elements.
    Let \(g_{2n} = \mathbf 1_{\{0\}}\) and \(g_{2n+1} = \mathbf 1_{\{1\}}\) for each \(n \in \omega\).
    Observe that \(W_{2n} = g_{2n}^{-1}(-2^{-2n}, 2^{-2n}) = \omega \setminus \{ 0 \}\) and
    \(W_{2n+1} = g_{2n_1}^{-1}(-2^{-2n-1}, 2^{-2n-1}) = \omega \setminus \{ 1 \}\)
    so
    \(V_{2n} = W_{2n} \setminus \{ \gamma(W_{2n} \setminus F_{2n}) \} = \omega \setminus \{ 0, p_{2n} \}\)
    and \(V_{2n+1} = W_{2n+1} \setminus \{ \gamma(W_{2n+1} \setminus F_{2n+1}) \}
    = \omega \setminus \{ 1, p_{2n+1} \}\)
    Since the \(\{ p_n : n \in \omega \}\) are pairwise distinct, \(\{ V_n : n \in \omega \}\) is a large cover
    of \(\omega\).

    Since \(\{ g_n : n \in \omega \}\) is a doubleton, it is clearly strongly closed discrete.
    Note also that, since \(F_n \cap \{ 0,1 \} = \varnothing\) for each \(n \in \omega\) by construction,
    \(g_n \in [\mathbf 0; F_n , 2^{-n}]\) for each \(n \in \omega\).
    Hence, Process A is respected, but
    \(\{ V_n : n \in \omega \}\) is a large cover of \(\omega\).
\end{proof}

\section{Preliminary Definitions and Conventions} \label{sec:Definitions}

For a set \(X\), we will observe the convention throughout that
\([X]_+^{<\aleph_0}\) consists of all \emph{nonempty} finite subsets of \(X\)
to avoid particular trivialities.
Also, when we have occasion to discuss finite-open style games, we will use
the notation \(X_{\mathrm{fin}}\) in place of \([X]_+^{<\aleph_0}\) to simplify
the game notation.

We use the word \emph{space} to mean a \emph{nonempty} topological space.
We make no general assumptions about separation axioms throughout and explicitly
state separation axioms as we assume them.
As usual, any undefined terms are to be understood as in \cite{Engelking}.

\subsection{Cover Types and Selection Principles}

For a space \(X\), we let \(\mathscr T_X\) denote the set of all proper and nonempty open subsets of \(X\)
to avoid certain trivialities, which we will elaborate on below.
Consequently, all open covers are assumed to be \emph{nontrivial}; in particular, we let \(\mathcal O_X\)
denote the set of all \(\mathscr U \subseteq \mathscr T_X\) such that \(X = \bigcup \mathscr U\).
We will also consider other cover types prevalent in the selection principles literature:
\begin{itemize}[leftmargin=2em]
    \item 
    An \emph{\(\omega\)-cover} of \(X\) is a cover \(\mathscr U \in \mathcal O_X\) such that, for each
    finite \(F \subseteq X\), there exists \(U \in \mathscr U\) such that \(F \subseteq U\).
    We let \(\Omega_X\) denote the set of all \(\omega\)-covers of \(X\).
    \item 
    A \emph{\(k\)-cover} of \(X\) is a cover \(\mathscr U \in \mathcal O_X\) such that,
    for each compact \(K \subseteq X\), there exists \(U \in \mathscr U\) such that
    \(K \subseteq U\).
    We let \(\mathcal K_X\) denote the set of all \(k\)-covers of \(X\).
    \item
    A \emph{large cover} of \(X\) is a cover \(\mathscr U \in \mathcal O_X\) such that,
    for each \(x \in X\), \(\{ U \in \mathscr U : x \in U \}\) is infinite.
    We let \(\Lambda_X\) denote the set of all large covers of \(X\).
\end{itemize}

We remark that, for any space \(X\),
\[\mathcal K_X \subseteq \Omega_X \subseteq \Lambda_X \subseteq \mathcal O_X,\]
which we will use throughout without additional mention.
We refer the reader to \cite[Proposition 2.2]{Chiozini} and \cite[Lemma 4]{CHContinuousFunctions}
for a proof of the inclusion \(\Omega_X \subseteq \Lambda_X\),
though the latter reference is phrased in terms of covers of ideal bases.

We recall the usual selection principles.
For more details on selection principles and relevant references, see
\cite{COOC1,KocinacSelectedResults,ScheepersSelectionPrinciples,ScheepersNoteMat}.
\begin{definition}
    Let \(\mathcal A\) and \(\mathcal B\) be sets.
    Then the single- and finite-selection principles are defined, respectively, to be the properties
    \[\mathsf S_1(\mathcal A, \mathcal B) \equiv 
    \left(\forall A \in \mathcal A^\omega\right)\left(\exists B \in \prod_{n \in \omega} A_n\right)\ \{B_n : n \in \omega\} \in \mathcal B\]
    and
    \[\mathsf S_{\mathrm{fin}}(\mathcal A, \mathcal B) \equiv 
    \left(\forall A \in \mathcal A^\omega\right)\left(\exists B \in \prod_{n \in \omega} [A_n]_+^{<\omega}\right)\ \bigcup\{B_n : n \in \omega\} \in \mathcal B.\]
    Following \cite{ScheepersNoteMat}, for a space \(X\) and topological operators \(\mathcal A\) and \(\mathcal B\),
    we write \(X \models \mathsf S_\ast(\mathcal A, \mathcal B)\), where \(\ast \in \{ 1 , \mathrm{fin} \}\),
    to mean that \(X\) satisfies the selection principle \(\mathsf S_\ast(\mathcal A_X, \mathcal B_X)\).
\end{definition}
Using this notation, recall that a space \(X\) is \emph{Menger} (resp. \emph{Rothberger})
if \(X \models \mathsf S_{\mathrm{fin}}(\mathcal O, \mathcal O)\)
(resp. \(X \models \mathsf S_1(\mathcal O, \mathcal O)\)).

\begin{definition}
    We say that a space \(X\) is
    \begin{itemize}[leftmargin=2em]
        \item 
        \emph{\(\lambda\)-Lindel\"{o}f} if every large cover of \(X\) has a countable subset
        which is also a large cover of \(X\).
        \item 
        \emph{\(\lambda\)-Rothberger} if \(X \models \mathsf S_1(\Lambda,\Lambda)\);
        that is, for every sequence \(\langle \mathscr U_n : n \in \omega \rangle\)
        of large covers of \(X\), there exists \(\langle U_n : n \in \omega \rangle \in \prod_{n\in\omega} \mathscr U_n\)
        such that \(\{ U_n : n\in \omega\}\) is a large cover of \(X\).
        \item 
        \emph{\(\lambda\)-Menger} if \(X \models \mathsf S_{\mathrm{fin}}(\Lambda, \Lambda)\);
        that is, for every sequence \(\langle \mathscr U_n : n \in \omega \rangle\)
        of large covers of \(X\), there exists \(\langle \mathscr F_n : n \in \omega \rangle \in \prod_{n\in\omega} \left[\mathscr U_n\right]_+^{<\aleph_0}\)
        such that \(\bigcup\{ \mathscr F_n : n\in \omega\}\) is a large cover of \(X\).
    \end{itemize}
\end{definition}

For a space \(X\) and \(x \in X\), we let \(\mathcal N_x = \{ U \in \mathscr T_X : x \in U \}\)
and \(\mathcal N_x^+ = \mathcal N_x \cup \{ X \}\).
For \(A \subseteq X\), \(\mathcal N_A = \{ U \in \mathscr T_X : A \subseteq U \}\)
and \(\mathcal N_A^+ = \mathcal N_A \cup \{ X \}\).
For a collection \(\mathscr A\) of subsets of \(X\), we will let
\[\mathcal N[\mathscr A] = \{ \mathcal N_A : A \in \mathscr A \}\]
where we understand that, when \(x \in X\), \(\mathcal N_{\{x\}} = \mathcal N_x\).
\begin{definition}
    We say that a space \(X\) is
    \begin{itemize}[leftmargin=2em]
        \item 
        \emph{topologically finite} if there exists a finite
        set \(F \subseteq X\) such that
        \[X = \bigcup \left\{ \bigcap \mathcal N_x^+ : x \in F \right\}.\]
        \item 
        \emph{topologically countable} if there exists a countable
        set \(A \subseteq X\) such that
        \[X = \bigcup \left\{ \bigcap \mathcal N_x^+ : x \in A \right\}.\]
    \end{itemize}
\end{definition}
\begin{lemma} \label{lem:SillyNhoodThing}
    For \(A \subseteq X\),
    \[\bigcap \mathcal N_A^+ = \bigcup \left\{ \bigcap \mathcal N_x^+ : x \in A \right\}.\]
\end{lemma}
\begin{proof}
    If \(y \in \bigcup \left\{ \bigcap \mathcal N_x^+ : x \in A \right\}\), let \(x \in A\) be such that
    \(y \in \bigcap \mathcal N_x^+\).
    Then, for any \(U \in \mathcal N_A^+\), note that \(x \in A \subseteq U\) implies that \(U \in \mathcal N_x^+\).
    Hence, \(y \in U\) which, since \(U \in \mathcal N_A^+\) was arbitrary, shows that
    \(y \in \bigcap \mathcal N_A^+\).

    On the other hand, if \(y \not\in \bigcup \left\{ \bigcap \mathcal N_x^+ : x \in A \right\}\),
    then, for each \(x \in A\), we can choose \(U_x \in \mathcal N_x^+\) such that
    \(y \not\in U_x\).
    Note that \(y \not\in \bigcup \{ U_x : x \in A \} \in \mathcal N_A^+\).
    It follows that \(y \not\in \bigcap \mathcal N_A^+\).
\end{proof}
\begin{proposition}
    For a space \(X\), \(\Lambda_X = \varnothing\) if and only if there is some \(x \in X\)
    such that \(\mathcal N_x\) is finite.
\end{proposition}
\begin{proof}
    If \(\mathscr U \in \Lambda_X\), then, for \(x \in X\),
    \(\{ U \in \mathscr U : x \in U \} \subseteq \mathcal N_x\), establishing that
    \(\mathcal N_x\) is infinite.
    On the other hand, if \(\mathcal N_x\) is infinite for every \(x\in X\),
    then \(\mathscr U := \bigcup \{ \mathcal N_x : x \in X \} \in \Lambda_X\).
\end{proof}

Clearly, any space with a finite topology admits no large covers.
On the other hand, any infinite \(T_1\) space has large covers since,
for every \(x \in X\), \(\left\{ X \setminus \{ y \} : y \in X \setminus \{x\} \right\}\)
is an infinite family of open neighborhoods of \(x\).

We note that simply having an infinite topology is not enough to guarantee
the existence of large covers.
\begin{example}
    Note that, \(X\), the excluded point topology on a countably infinite set,
    \cite[\hyperlink{https://topology.pi-base.org/spaces/S000012}{S12}]{PiBase},
    has the property that \(\Lambda_X = \varnothing\) but \(\mathscr T_X\) is
    infinite.
    Indeed, the excluded point has only one neighborhood, but every other point has
    an infinite open neighborhood filter.
\end{example}
\begin{proposition} \label{prop:SillyNoOmegaThing}
    For any space \(X\), \(\Omega_X = \varnothing\)
    if and only if \(X\) is topologically finite.
\end{proposition}
\begin{proof}
    If \(\Omega_X = \varnothing\), there must be some \(F \in [X]_+^{<\aleph_0}\)
    such that \(\mathcal N_F^+ = \{X\}\).
    Hence, by Lemma \ref{lem:SillyNhoodThing},
    \[X = \bigcap \mathcal N_F^+ = \bigcup \left \{ \bigcap \mathcal N_x^+ : x \in F \right\}.\]

    On the other hand, suppose \(F \subseteq X\) is finite with the property that
    \[X = \bigcup \left \{ \bigcap \mathcal N^+_x : x \in F \right \}.\]
    Now, suppose \(U \subseteq X\) is open with \(F \subseteq U\).
    Note that, for each \(x \in F\), \(x \in \bigcap \mathcal N_x^+ \subseteq U\).
    It follows that
    \[X = \bigcup \left \{ \bigcap \mathcal N^+_x : x \in F \right \} \subseteq U.\]
    Hence, \(\Omega_X = \varnothing\).
\end{proof}
\begin{proposition}
    For any space \(X\), \(\mathcal K_X = \varnothing\)
    if and only if \(X\) is compact.
\end{proposition}
\begin{proof}
    If \(X\) is compact, there are evidently no \(k\)-covers since any such cover
    would have to be trivial.
    On the other hand, suppose \(X\) is not compact and let \(\mathscr U\) be an open cover
    of \(X\) that admits no finite subcover.
    Then, for every nonempty compact \(K \subseteq X\), let \(\mathscr V_K\) be a finite subset of
    \(\mathscr U\) such that \(K \subseteq \bigcup \mathscr V_K\).
    Note that \(\bigcup \mathscr V_K \in \mathscr T_X\).
    Then \(\mathscr W = \left\{ \bigcup \mathscr V_K : K \in \mathsf K(X) \right\}\),
    where \(\mathsf K(X)\) denotes the set of all nonempty compact subsets of \(X\),
    is a \(k\)-cover of \(X\).
\end{proof}
\begin{example}
    Note that \(X = \omega+1\) with its usual order topology is an infinite compact Hausdorff space.
    So \(\mathcal K_X = \varnothing\) but \(\Omega_X \neq \varnothing\).
\end{example}
\begin{example} \label{example:SillyExample}
    There is a space \(X\) such that \(\Omega_X = \varnothing\) and \(\Lambda_X \neq \varnothing\).
    Consider the disjoint union of two copies of the excluded point topology on a countably
    infinite set.
    In particular, consider the set \(\omega \times 2\) and let \(A_0 = \omega \times \{0\}\),
    \(A_1 = \omega \times \{1\}\), \(p_0 = ( 0,0 )\), and \(p_1 = ( 0 , 1 )\).
    Then let \(X\) be \(\omega \times 2\) with the topology
    \[\{ E \subseteq X : (p_0 \in E \to A_0 \subseteq E)
    \wedge (p_1 \in E \to A_1 \subseteq E) \}.\]
    Observe that \(X\) is topologically finite since
    \[X = A_0 \cup A_1 = \left(\bigcap \mathcal N_{p_0} \right) \cup \left( \bigcap \mathcal N_{p_1} \right).\]
    So \(\Omega_X = \varnothing\).
    However, note that
    \[
    \{ A_0 \cup \{ ( j , 1 ) \} : j \geq 1 \} \} \cup \{ A_1 \cup \{ ( j , 0 ) \} : j \geq 1 \} \} \in \Lambda_X,
    \]
    so \(\Lambda_X \neq \varnothing\).
\end{example}

\subsection{Other Topological Objects and Selection Games}

As usual, we use \(C_p(X)\) to denote the space of all continuous real-valued functions defined
on \(X\) endowed with the topology of pointwise convergence.
For a finite subset \(F\) of \(X\), \(f \in C_p(X)\), and \(\varepsilon > 0\),
we let
\[[f ; F, \varepsilon] = \{ g \in C_p(X) : (\forall x \in F)\ |f(x) - g(x)| < \varepsilon \}.\]
We note that these sets constitute the standard basis for \(C_p(X)\) and that the sets
of the form \([f; \{x\}, \varepsilon]\) constitute a subbasis for \(C_p(X)\).

A space \(X\) is said to be \emph{discretely selective} if, given any sequence \(\langle U_n : n \in \mathbb N \rangle\)
of nonempty open subsets of \(X\), there exists a selection \(x_n \in U_n\) for each \(n \in \mathbb N\) such that
\(\{ x_n : n \in \mathbb N \}\) is a closed and relatively discrete subset of \(X\).
The property of being discretely selective first appeared in
\cite[Lemma 3.8.1]{GuerreroSanchezTkachukCp} and later gained its name in \cite{TkachukClosedDiscrete}.
There is a naturally associated game of length \(\omega\) introduced in \cite{TkachukTwoPointPickingGames}
where, at round \(n \in \omega\),
the first player chooses some \(U_n \in \mathscr T_X\) and the second player responds with
\(x_n \in U_n\).
The second player is declared the winner if \(\{ x_n : n \in \omega \}\) is closed and relatively
discrete in \(X\); otherwise, the first player wins.
This game is known as the \emph{closed discrete game} and, aside from its introduction in \cite{TkachukTwoPointPickingGames},
has been expanded upon in, for example, \cite{ClontzHolshouser,CHCompactOpen,CHContinuousFunctions,CExcursions}.
The connections established in \cite{TkachukTwoPointPickingGames} and \cite{ClontzHolshouser}
between certain topological games on a space \(X\) and other topological games on \(C_p(X)\)
inspired the following notion, due to Chiozini.

\begin{definition}[\cite{Chiozini}]
    A set \(A \subseteq C_p(X)\) is said to be \emph{strongly closed discrete} if,
    for every \(g \in C_p(X)\), there exist \(x \in X\) and \(\varepsilon > 0\)
    such that \(A \cap [g; \{x\}, \varepsilon]\) is finite.
    Then we let \(s\mathrm{CD}_{C_p(X)}\) denote the set of all
    strongly closed discrete subsets of \(C_p(X)\).
\end{definition}
We note here that the property of being strongly closed discrete could be extended
to the class of all topological spaces by phrasing it as the property
of being closed discrete with respect to a given subbase compatible with the topology.
As such, the strongly closed discrete game can be related to subbasic selection games
as discussed in \cite{GuerreroSanchezTkachuk2017},
which we'll elaborate on in Section \ref{section:CoveringPrinciples}.

For the reader's convenience, we now define traditional selection games for two players, P1 and P2, of length \(\omega\).
\begin{definition}
    Given sets \(\mathcal A\) and \(\mathcal B\), we define the \emph{finite-selection game}
    \(\mathsf{G}_{\mathrm{fin}}(\mathcal A, \mathcal B)\) for \(\mathcal A\) and \(\mathcal B\) as follows.
    In round \(n \in \omega\), P1 plays \(A_n \in \mathcal A\) and P2 responds with \(\mathscr F_n \in [A_n]_+^{<\omega}\).
    We declare P2 the winner if \(\bigcup\{ \mathscr F_n : n \in \omega \} \in \mathcal B\).
    Otherwise, P1 wins.

    We analogously define the \emph{single-selection game}
    \(\mathsf{G}_{1}(\mathcal A, \mathcal B)\) for \(\mathcal A\) and \(\mathcal B\) as follows.
    In round \(n \in \omega\), P1 plays \(A_n \in \mathcal A\) and P2 responds with \(x_n \in A_n\).
    We declare P2 the winner if \(\{x_n : n \in \omega \} \in \mathcal B\).
    Otherwise, P1 wins.
\end{definition}
For a class \(\mathcal B\), we will use \(\neg \mathcal B\) to denote the complementary class
of \(\mathcal B\).
Hence, in our notation, the games \(\Lambda FO(X)\) and \(sCD(C_p(X))\) of \cite{Chiozini}
correspond to the games \(\mathsf G_1(\mathcal N[X_{\mathrm{fin}}], \neg \Lambda_X)\)
and
\(\mathsf G_1(\mathscr T_{C_p(X)} , s\mathrm{CD}_{C_p(X)})\), respectively.
As such, we will refer to the game \(\mathsf G_1(\mathscr T_{C_p(X)} , s\mathrm{CD}_{C_p(X)})\)
as the \emph{strong closed discrete game} on \(C_p(X)\).

We remark now one reason we take \(\mathscr T_X\) to be as defined, which disallows the second player
from playing the entire space \(X\) in variations of the point-open game, like
\(\mathsf G_1(\mathcal N[X_{\mathrm{fin}}], \neg \Lambda_X)\).
In the usual point-open game of Galvin and Telg\'{a}rsky \cite{Galvin1978,Telgarsky1975},
the second player has to necessarily avoid playing
\(X\) since they wish to avoid covering the space.
In the game \(\mathsf G_1(\mathcal N[X_{\mathrm{fin}}], \neg \Lambda_X)\), however,
if the second player were allowed to play \(X\) in every round, then they would successfully avoid
producing a large cover of \(X\).
So, to guarantee that the game \(\mathsf G_1(\mathcal N[X_{\mathrm{fin}}], \neg \Lambda_X)\)
is nontrivial, we must restrict the second player's available moves.

Games have naturally associated notions of strategies.
For the types of strategies appearing in this paper, see \cite[Definition 3.7]{CCHMengerRothbergerSurvey};
additionally, we will briefly mention \emph{tactical strategies}
(\cite[Definition 11]{ClontzDualSelection})
for P2 in Section \ref{section:CoveringPrinciples}.
In the present notation, a tactical strategy for P2 in the game
\(\mathsf G_1(\mathcal A, \mathcal B)\) is a choice function
\(\tau : \mathcal A \to \bigcup \mathcal A\);
that is, \(\tau(\mathscr U) \in \mathscr U\) for each \(\mathscr U \in \mathcal A\).
A tactical strategy is said to be \emph{winning} if, given any sequence
\(\langle \mathscr U_n : n \in \omega \rangle \in \mathcal A^\omega\),
\(\{ \tau(\mathscr U_n) : n \in \omega \} \in \mathcal B\).

To assert that one of the players has a particular kind of winning strategy for
a given game, we will use the notation provided in \cite[Notation 12]{ClontzDualSelection}.

\begin{definition}
    For any topological property \(\mathcal P\) defined in terms of selection principles,
    we say that a space is \emph{Markov \(\mathcal P\)} if the second player of the
    corresponding selection game has a winning Markov strategy.
\end{definition}
So, for example, we say that a space is Markov \(\lambda\)-Menger
if P2 has a winning Markov strategy in the game \(\mathsf G_{\mathrm{fin}}(\Lambda_X, \Lambda_X)\).

We will say that two games are \emph{dual} in the sense of \cite[Definition 19]{ClontzDualSelection},
which includes three strategic strengths, and use the notation
\[\mathsf G_1(\mathcal A, \mathcal B) \dualarrows \mathsf G_1(\mathcal C, \mathcal D)\]
to assert that \(\mathsf G_1(\mathcal A, \mathcal B)\) and \(\mathsf G_1(\mathcal C, \mathcal D)\)
are dual in this sense.
Note particularly that this notion of duality is stronger than the typical use of duality,
which is usually restricted to duality only in terms of perfect-information strategies.
For the granular notions of duality, like \emph{perfect-} and \emph{Markov-information duality},
see Definitions 16, 17, and 18 of \cite{ClontzDualSelection}.
We will also use Corollaries 29 and 31 of \cite{ClontzDualSelection} without additional mention,
which, respectively, assert that, for any space \(X\) and any set \(\mathcal B\),
\[\mathsf G_1(\mathcal O_X, \mathcal B) \dualarrows \mathsf G_1(\mathcal N[X], \neg \mathcal B)
\text{ and }\mathsf G_1(\Omega_X, \mathcal B) \dualarrows \mathsf G_1(\mathcal N[X_{\mathrm{fin}}], \neg \mathcal B).\]

Finally, we recall the partial orderings \(\leq_{\mathrm{II}}\) and \(\leq_{\mathrm{II}}^+\)
between selection games and refer the reader to \cite[Definition 3.8]{CCHMengerRothbergerSurvey}
for the particular implications constituting the orderings;
\(\leq_{\mathrm{II}}\) is the conjunction of four implications across four levels of strategic type
and \(\leq_{\mathrm{II}}^+\) extends \(\leq_{\mathrm{II}}\) to include an implication for
constant strategies for the first player.
For game equivalences, just as in \cite[Definition 2.14]{CHVietoris}, we use the notation
\[\mathsf G_1(\mathcal A, \mathcal B) \equiv \mathsf G_1(\mathcal C, \mathcal D)\] if
\(\mathsf G_1(\mathcal A, \mathcal B) \leq_{\mathrm{II}} \mathsf G_1(\mathcal C, \mathcal D)\)
and \(\mathsf G_1(\mathcal C, \mathcal D) \leq_{\mathrm{II}} \mathsf G_1(\mathcal A, \mathcal B)\);
similarly, we use the notation
\[\mathsf G_1(\mathcal A, \mathcal B) \leftrightarrows \mathsf G_1(\mathcal C, \mathcal D)\] if
\(\mathsf G_1(\mathcal A, \mathcal B) \leq_{\mathrm{II}}^+ \mathsf G_1(\mathcal C, \mathcal D)\)
and \(\mathsf G_1(\mathcal C, \mathcal D) \leq_{\mathrm{II}}^+ \mathsf G_1(\mathcal A, \mathcal B)\).
Note here that both of these notions of game equivalence are stronger than the typical sense of
game equivalence, which is usually concerned only with equivalence for perfect-information strategies.
We also highlight the reason neither \(\leq_{\mathrm{II}}\) nor \(\leq_{\mathrm{II}}^+\) include
implications for the level of tactical strategies.
A straightforward argument establishes that P2 has a winning tactical strategy in the Menger game
on \(X\) if and only if \(X\) is compact, which is typically not of interest in the general theory
of Menger-type covering properties.
A similarly straightforward argument shows that
P2 has a winning tactical strategy in the Rothberger game on \(X\) if and only if
\(X\) has no nontrivial open covers!
It is immediately evident that spaces lacking any nontrivial open covers are not of much
interest in any context, though we note that the equivalences of Theorem \ref{thm:SingleLambdaOpen}
could be extended to equivalences involving the level of tactical strategies because of this
triviality.

\section{Generalized Covering Principles} \label{section:CoveringPrinciples}

For any collection \(\mathscr A \subseteq \wp(X)\), we can define a \emph{cover by elements of \(\mathscr A\)}
to be a collection \(\mathscr C \subseteq \mathscr A\) such that \(X = \bigcup \mathscr C\).
Let \(\mathscr A_{\mathrm{cov}}\) denote the collection of all such covers.
We can then define the corresponding notions of being \(\mathscr A\)-compact,
\(\mathscr A\)-Lindel\"{o}f, \(\mathscr A\)-Rothberger, etc.\footnote{We recognize that this overlaps
with common notations like, for example, \(\omega\)-Rothberger, but trust that context will dispel any
potential confusion.}

When \(\mathscr S\) is a subbase for the space \(X\), the Alexander Subbase Lemma
is the assertion that \(X\) is compact if and only if \(X\) is \(\mathscr S\)-compact.
Recall that a collection \(\mathscr S \subseteq \mathscr T_X\) is a subbase
compatible with the topology on \(X\) if
\[\left\{ \bigcap \mathscr F : \mathscr F \in \left[ \mathscr S \right]^{<\aleph_0} \right\}\]
is a basis for its topology.
Note that any basis on a topology is also a subbasis, so, any space which fails to be
Menger, Rothberger, or Lindel\"{o}f fails to be \(\mathscr B\)-Menger, \(\mathscr B\)-Rothberger,
or \(\mathscr B\)-Lindel\"{o}f, respectively, for any basis \(\mathscr B\) compatible with its
topology.

We note that, for a subbase \(\mathscr S\) of a space \(X\), the corresponding \(\mathscr S\)-Rothberger game
was studied by Guerrero S\'{a}nchez and Tkachuk \cite{GuerreroSanchezTkachuk2017}
with the notation \(CE(\mathscr S,X)\);
in our notation, the game \(CE(\mathscr S, X)\) is \(\mathsf G_1(\mathscr S_{\mathrm{cov}},\mathscr S_{\mathrm{cov}})\).
We also note that, in that same paper, the authors demonstrated
that the \(\mathscr S\)-Rothberger game is perfect-information dual to the primary game of interest in the paper,
the point-open subbase game, \(PO(\mathscr S,X)\); in our notation, \(PO(\mathscr S,X)\) is equivalent to
\(\mathsf G_1(\mathscr S_{\mathrm{nh}}[X], \neg \mathscr S_{\mathrm{cov}})\), where\footnote{We use the string \emph{nh}
to evoke the phrase \emph{neighborhoods}.}
\[\mathscr S_{\mathrm{nh}}[X] = \{ \mathscr S_{\mathrm{nh}}(x) : x \in X \}
= \{ \{ A \in \mathscr S : x \in A \} : x \in X \}.\]
By using standard duality techniques recorded by Clontz \cite{ClontzDualSelection},
we presently extend the perfect-information duality of \cite[Theorem 3.12]{GuerreroSanchezTkachuk2017}
to a more general one which includes, in addition, limited-information strategies.
Before we state the generalization, extend the notation above as follows:
for a collection \(\mathscr A\) of subsets of \(X\), let\footnote{Since the collection \(\mathscr A\)
may have nothing to do with the underlying topology, we use the string \emph{pf} here
to simply evoke the phrase \emph{point filter}.}
\[\mathscr A_{\mathrm{pf}}(x) = \{ A \in \mathscr A : x \in A \}\]
for each \(x \in X\) and
\[\mathscr A_{\mathrm{pf}}[X] = \{ \mathscr A_{pf}(x) : x \in X \}.\]
\begin{theorem} \label{thm:SubbasicDuality}
    For any space \(X\), any collection \(\mathscr A\) of subsets of \(X\), and any set \(\mathcal B\),
    the two games
    \(\mathsf G_1(\mathscr A_{\mathrm{cov}},\mathcal B)\) and
    \(\mathsf G_1(\mathscr A_{\mathrm{pf}}[X], \neg \mathcal B)\)
    are perfect-, Markov-, and tactical-information dual.
\end{theorem}
\begin{proof}
    The statement of the theorem will follow from \cite[Corollary 26]{ClontzDualSelection}
    if we show that \(\mathscr A_{\mathrm{pf}}[X]\) is a reflection (\cite[Definition 5]{ClontzDualSelection})
    of \(\mathscr A_{\mathrm{cov}}\).
    Note that a choice \(f\) of \(\mathscr A_{\mathrm{pf}}[X]\) is a function which
    assigns, for each \(\mathscr A_{\mathrm{pf}}(x)\), some \(f(x) \in \mathscr A_{\mathrm{pf}}(x)\).
    Clearly, then, for any choice \(f\) of \(\mathscr A_{\mathrm{pf}}[X]\), \(\mathrm{range}(f) \in \mathscr A_{\mathrm{cov}}\).
    Moreover, if \(\mathscr C \in \mathscr A_{\mathrm{cov}}\), we can, for each \(x \in X\),
    choose \(f(x) \in \mathscr C\) with \(x \in f(x)\).
    Note that \(f(x) \in \mathscr A_{\mathrm{pf}}(x)\) since \(x \in f(x) \in \mathscr C \subseteq \mathscr A\).
    Hence, \(f\) is a choice on \(\mathscr A_{\mathrm{pf}}[X]\) with the property that
    \(\mathrm{range}(f) \subseteq \mathscr C\).
    Therefore, \(\mathscr A_{\mathrm{pf}}[X]\) is a reflection
    of \(\mathscr A_{\mathrm{cov}}\), and the proof is complete.
\end{proof}
\begin{corollary}
    For any space \(X\) and any subbase \(\mathscr S\) for \(X\), the two games
    \(\mathsf G_1(\mathscr S_{\mathrm{cov}},\mathscr S_{\mathrm{cov}})\) and
    \(\mathsf G_1(\mathscr S_{\mathrm{nh}}[X], \neg \mathscr S_{\mathrm{cov}})\)
    are perfect-, Markov-, and tactical-information dual.
\end{corollary}

In some very important sense, when relativizing covering properties to ones involving subbasic covers,
the Alexander Subbase Lemma is optimal.
The standard proof of the Alexander Subbase Lemma employing Zorn's Lemma cannot be extended
to the context where \emph{compact} is replaced with \emph{Lindel\"{o}f}.
\begin{example}
    Let \(\mathbf\Gamma\) be the Niemytzki plane.
    Let \(H = \{ (x,y) \in \mathbf \Gamma : y > 0\}\) and, for each \(x \in \mathbb R\),
    let \[B_x = B((x,1);1) \cup \{(x,0)\}.\]
    Then \[\mathscr U_0 := \{ H \} \cup \{ B_x : x \in \mathbb R\}\]
    is an open cover of \(\mathbf \Gamma\) with no countable subcover.
    For each \(n \in \omega\), let \(S_n = \bigcup \{ B_x : x < n \}\).
    Then, for each \(n \in \omega\), recursively define
    \(\mathscr U_{n+1} = \mathscr U_n \cup \{ S_n \}\).
    Note that \(\mathscr U_n\) is an open cover of \(\mathbf \Gamma\) with
    no countable subcover.
    Then \(\{ \mathscr U_n : n \in \omega \}\) is a chain (ordered by inclusion)
    of open covers of \(\mathbf \Gamma\) with no upperbound in the poset of all
    open covers of \(\mathbf \Gamma\) admitting no countable subcover.
    Indeed, note that
    \[\{H\} \cup \{ S_n : n \in \omega \} \subseteq \bigcup_{n\in\omega} \mathscr U_n\]
    covers \(\mathbf \Gamma\).
\end{example}
Beyond the bounded chain issue mentioned above, we have an explicit example of a space
which is \(\mathscr S\)-Lindel\"{o}f but not Lindel\"{o}f.
\begin{example}
    Let \(\mathbb R_\ell\) be the Sorgenfrey line and consider the Sorgenfrey plane
    \(\mathbb R_\ell^2\).
    Consider the subbasis \(\mathscr S = \mathscr V \cup \mathscr H\) for \(\mathbb R_\ell^2\)
    where \[\mathscr V = \left\{ [a,b) \times \mathbb R_\ell : a , b \in \mathbb R, a < b \right\}\]
    and \[\mathscr H = \left\{ \mathbb R_\ell \times [a, b) : a, b \in \mathbb R, a < b \right\}.\]
    Then \(\mathbb R_\ell^2\) is \(\mathscr S\)-Lindel\"{o}f.
\end{example}
\begin{proof}
    Let \(\mathscr U \subseteq \mathscr S\) be a cover of \(\mathbb R_\ell^2\) and let
    \(\mathscr U_v = \{ U \in \mathscr U : U \in \mathscr V\}\) and
    \(\mathscr U_h = \{ U \in \mathscr U : U \in \mathscr H \}\).
    Then write
    \[\mathscr U_v = \{ [a_\lambda, b_\lambda) \times \mathbb R_\ell : \lambda \in \Lambda_v \}\]
    and
    \[\mathscr U_h = \{ \mathbb R_\ell \times [a_\lambda, b_\lambda) : \lambda \in \Lambda_h \}\]
    where \(\Lambda_v \cap \Lambda_h = \varnothing\).
    Note that
    \[W = \bigcup \left\{(a_\lambda, b_\lambda) \times \mathbb R_\ell : \lambda \in \Lambda_v \right\}
    \cup \bigcup \left\{ \mathbb R_\ell \times (a_\lambda, b_\lambda) : \lambda \in \Lambda_h \right\}\]
    is an open subspace of \(\mathbb R^2\), the standard Euclidean plane.
    Hence, there are countable sets \(\Lambda_v^\prime \subseteq \Lambda_v\) and \(\Lambda_h^\prime \subseteq \Lambda_h\)
    such that
    \[W = \bigcup \left\{ (a_\lambda, b_\lambda) \times \mathbb R_\ell : \lambda \in \Lambda_v^\prime \right\}
    \cup \bigcup \left\{ \mathbb R_\ell \times (a_\lambda, b_\lambda) : \lambda \in \Lambda_h^\prime \right\}.\]

    Now, for \(x \in \mathbb R\), let \(V_x = \{x\} \times \mathbb R_\ell\) and \(H_x = \mathbb R_\ell \times \{x\}\).
    We claim that there exist countable sets \(X_1, X_2 \subseteq \mathbb R\) such that
    \[\mathbb R_\ell^2 \setminus W \subseteq \bigcup_{x \in X_1} V_x \cup \bigcup_{y \in X_2} H_y.\]

    Note that, for each \((x,y) \in \mathbb R_\ell^2 \setminus W\),
    there is either some \(\lambda \in \Lambda_v\) such that \(x = a_\lambda\) or some
    \(\lambda \in \Lambda_h\) such that \(y = a_\lambda\).
    Let
    \[M_v = \left\{ x \in \mathbb R : (\exists y \in \mathbb R)(\exists \lambda \in \Lambda_v)\ 
    \left[ (x,y) \not\in W \wedge x = a_\lambda \right] \right\}\]
    and
    \[M_h = \left\{ y \in \mathbb R : (\exists x \in \mathbb R)(\exists \lambda \in \Lambda_h)\ 
    \left[ (x,y) \not\in W \wedge y = a_\lambda \right] \right\}.\]
    
    For each \(x \in M_v\), choose \(\lambda_x \in \Lambda_v\) such that \(x = a_{\lambda_x}\)
    and note that \[(a_{\lambda_x} , b_{\lambda_x}) \times \mathbb R_\ell \subseteq W.\]
    Choose \(q_x \in \mathbb Q\) with \(a_{\lambda_x} < q_x < b_{\lambda_x}\).
    This defines a map \(M_v \to \mathbb Q\).

    Similarly, for each \(y \in M_h\), choose \(\lambda_y \in \Lambda_h\) such that
    \(y = a_{\lambda_y}\).
    Note that \[\mathbb R_\ell \times (a_{\lambda_y}, b_{\lambda_y}) \subseteq W.\]
    Then choose \(r_y \in \mathbb Q\) with \(a_{\lambda_y} < r_y < b_{\lambda_y}\).
    This defines a map \(M_h \to \mathbb Q\).

    We show that both mappings \(M_v \to \mathbb Q\) and \(M_h \to \mathbb Q\) are injective.
    Suppose \(x_1,x_2 \in M_v\), \(x_1 < x_2\).
    Let \(y_1, y_2 \in \mathbb R\) be such that \((x_1,y_1) \not\in W\) and \((x_2,y_2) \not\in W\).
    Note that, for any \(x \in (x_1, q_{x_1}] \subseteq (a_{\lambda_{x_1}}, b_{\lambda_{x_1}})\) and any \(y \in \mathbb R\),
    \((x,y) \in W\).
    Hence, \(q_{x_1} < x_2 < q_{x_2}\).

    A similar argument establishes that the mapping \(M_h \to \mathbb Q\) is injective.

    Hence, both \(M_v\) and \(M_h\) are countable.
    It follows that
    \[\mathbb R_\ell^2 \setminus W \subseteq \bigcup_{x \in M_v} V_x \cup \bigcup_{y \in M_h} H_y,\]
    as claimed.

    Finally,
    \begin{align*}
        \left\{ [a_\lambda, b_\lambda) \times \mathbb R_\ell : \lambda \in \Lambda_v^\prime \right\}
        &\cup \left\{ \mathbb R_\ell \times [a_\lambda, b_\lambda) : \lambda \in \Lambda_h^\prime \right\}\\
        &\cup \left\{ [a_{\lambda_x}, b_{\lambda_x}) \times \mathbb R_\ell : x \in M_v \right\}\\
        &\cup \left\{ \mathbb R_\ell \times  [a_{\lambda_y}, b_{\lambda_y}) : y \in M_h \right\}
    \end{align*}
    is our desired countable subcover of \(\mathbb R_\ell^2\).
\end{proof}
We also note that a space can be \(\mathscr S\)-Rothberger without even being Menger.
\begin{example} \label{example:Irrationals}
    Consider \(\mathbb P = \mathbb R \setminus \mathbb Q\) and recall that \(\mathbb P\) is not Menger.
    Then consider the subbasis
    \(\mathscr S = \mathscr L \cup \mathscr R\) for \(\mathbb P\) where
    \[\mathscr L = \{ (\leftarrow, x) \cap \mathbb P : x \in \mathbb Q\}\]
    and \[\mathscr R = \{ (x, \rightarrow) \cap \mathbb P : x \in \mathbb Q\}.\]
    Then \(\mathbb P\) is \(\mathscr S\)-Rothberger.
\end{example}
\begin{proof}
    Let \(\mathscr U_n\) be a subbasic cover of \(\mathbb P\) for each \(n \in \omega\)
    and let \(D = \{ d_n : n \in \omega \}\) be a dense subset of \(\mathbb P\).
    For each \(n \in \omega\), choose some \(U_{n+1} \in \mathscr U_{n+1}\) such that
    \(d_n \in U_{n+1}\).
    
    If \(\mathbb P = W:= \bigcup \{ U_{n+1} : n \in \omega \}\), we can simply choose an arbitrary
    \(U_0 \in \mathscr U_0\) to complete the selection.
    Otherwise, \(\mathbb P \setminus W \neq\varnothing\).
    In such a case, we claim that the cardinality of \(\mathbb P \setminus W\) is one.
    Suppose \(x \in \mathbb P \setminus W\) and that \(y \in \mathbb P\) with \(x \neq y\).

    If \(x < y\), by density of \(D\), there is some \(n \in \omega\) such that \(x < d_n < y\).
    Then \(d_n \in U_{n+1}\) and, since \(x \not\in W\), there must be some \(b \in \mathbb Q\)
    such that \(U_{n+1} = (b, \rightarrow)\).
    Hence, \(b < d_n < y\) and we have that \(y \in U_{n+1} \subseteq W\).

    Otherwise, \(y < x\) and a similar argument asserts that \(y \in W\).

    To complete the selection, we need only choose some \(U_0\in\mathscr U_0\)
    such that \(x \in U_0\) where \(x\) is such that \(\{x\} = \mathbb P \setminus W\).
\end{proof}
By Example \ref{example:Irrationals}, we see that a space \(X\) can admit a subbase \(\mathscr S\)
for which \(X\) is \(\mathscr S\)-Rothberger even though it fails to be Rothberger.
For any such space \(X\), note that it fails to be \(\mathscr B\)-Rothberger for any basis
\(\mathscr B\) compatible with its topology.

However, in some cases, a space can fail to be \(\mathscr S\)-Rothberger with respect to any subbase \(\mathscr S\)
compatible with its topology.

The following example is a direct adaptation of \cite[Theorem 3.10]{GuerreroSanchezTkachuk2017},
with slightly modified language emphasizing the involvement of limited-information strategies.
\begin{example} \label{example:Measurable}
    Under the assumption that measurable cardinals exist, if \(\kappa\) is any cardinal
    at least as large as the least measurable cardinal,
    then \(D(\kappa)\) is not \(\mathscr S\)-Lindel\"{o}f for any subbase \(\mathscr S\) compatible
    with \(D(\kappa)\).
\end{example}
\begin{proof}
    We start by letting \(\mu\) be a \(\sigma\)-complete free ultrafilter on \(\kappa\)
    and \(\mathscr S\) be any subbase compatible with the topology of \(D(\kappa)\).
    Note that, for any \(x \in \kappa\), there is some \(\mathcal E_x \in [\mathscr S]^{<\aleph_0}\)
    such that \(\bigcap \mathcal E_x = \{x\}\).
    Since \(\{x\} \not\in \mu\), there must be some \(E_x \in \mathcal E_x\) with \(E_x \not\in \mu\).
    We claim that this is a winning tactical strategy for P2 in the game
    \(\mathsf G_1(\mathscr S_{\mathrm{nh}}[X], \neg \mathscr S_{\mathrm{cov}})\).
    Indeed, suppose \(\{ x_n : n \in \omega \} \subseteq \kappa\) and note that
    \(\{E_{x_n} : n \in \omega\} \subseteq \mathscr S\) has the property that \(x_n \in E_{x_n} \not\in \mu\)
    for each \(n \in \omega\).
    By \(\sigma\)-completeness, \(\bigcup \{ E_{x_n} : n \in \omega \} \not\in \mu\).
    Hence, \(\{ E_{x_n} : n \in \omega \}\) fails to cover \(D(\kappa)\),
    and so P2 has won.

    Finally, by Theorem \ref{thm:SubbasicDuality},
    \[\mathrm{II} \underset{\mathrm{tact}}{\uparrow} \mathsf G_1(\mathscr S_{\mathrm{nh}}[D(\kappa)], \neg \mathscr S_{\mathrm{cov}})
    \iff \mathrm{I} \underset{\mathrm{con}}{\uparrow} \mathsf G_1(\mathscr S_{\mathrm{cov}}, \mathscr S_{\mathrm{cov}}),\]
    which establishes that \(D(\kappa)\) is not \(\mathscr S\)-Lindel\"{o}f.
\end{proof}
With the limited-information strategy type clarification in Example \ref{example:Measurable},
we would like to update \cite[Question 4.1]{GuerreroSanchezTkachuk2017} to:
\begin{question}
    If \(\kappa\) is such that \(D(\kappa)\) is not \(\mathscr S\)-Lindel\"{o}f for any subbase
    \(\mathscr S\) compatible with \(D(\kappa)\), does \(\kappa\) necessarily admit a \(\sigma\)-complete
    free ultrafilter?
    Equivalently, if \(D(\kappa)\) has the property that P2 has a winning tactical strategy
    in the subbasic point-open game for each compatible subbasis, does
    \(\kappa\) necessarily admit a \(\sigma\)-complete
    free ultrafilter?
\end{question}
There are some fairly obvious candidates for such filters on \(\kappa\) using
these hypotheses, some of which are evidently \(\sigma\)-complete filters and others
which are evidently ultrafilters.
We have not, however, found a proof which demonstrates that such candidates
are both \(\sigma\)-complete and ultrafilters.

\section{Commentary on Large Covers} \label{section:LargeCovers}

For two classes \(\mathcal A\) and \(\mathcal B\), we use \(\binom{\mathcal A}{\mathcal B}\) to mean that,
for every \(\mathscr U \in \mathcal A\), there exists \(\varphi : \omega \to \mathscr U\)
such that \(\{ \varphi(n) : n \in \omega \} \in \mathcal B\).
For example, note that \(\binom{\mathcal O}{\mathcal O}\) is exactly the Lindel\"{o}f property.

\begin{lemma} \label{lem:CompactnessThing}
    Suppose \(\mathscr U\) is an open cover of a space \(X\).
    If \(\mathscr U\) admits no finite subcovers, then
    \[\mathscr W := \left\{ \bigcup \mathscr F : \mathscr F \in \left[ \mathscr U \right]_+^{<\aleph_0} \right\}\]
    is a \(k\)-cover of \(X\).
    
    Consequently, for any space \(X\), the following are equivalent.
    \begin{enumerate}[label=(\roman*)]
        \item \label{lolCompact}
        \(X\) is compact.
        \item \label{lolLambdaCompact}
        Every large cover of \(X\) admits a finite subcover.
        \item \label{lolOmegaCompact}
        Every \(\omega\)-cover of \(X\) admits a finite subcover.
        \item \label{lolkCompact}
        Every \(k\)-cover of \(X\) admits a finite subcover.
    \end{enumerate}
\end{lemma}
\begin{proof}
    It is evident that \ref{lolCompact}\(\implies\)\ref{lolLambdaCompact}\(\implies\)\ref{lolOmegaCompact}\(\implies\)\ref{lolkCompact}.
    To show that \ref{lolkCompact}\(\implies\)\ref{lolCompact}, suppose \(X\) fails to be compact
    and let \(\mathscr U\) be an open cover of \(X\) admitting no finite subcover.
    Note then that
    \[\mathscr W := \left\{ \bigcup \mathscr F : \mathscr F \in \left[ \mathscr U \right]_+^{<\aleph_0} \right\}\]
    is a \(k\)-cover of \(X\).
    Moreover, note that, for any finite \(\mathscr F \subseteq \mathscr W\), there is a finite \(\mathscr G \subseteq \mathscr U\)
    such that \(\bigcup \mathscr F = \bigcup \mathscr G\).
    Hence, \(\mathscr W\) is a \(k\)-cover of \(X\) which admits no finite subcover.
\end{proof}

Though the following is certainly well-known, we provide a proof here for the benefit of the reader.
\begin{lemma} \label{lem:BasicLambdaLindelof}
    For any space \(X\), the following are equivalent.
    \begin{enumerate}[label=(\roman*)]
        \item \label{BasicLindelof}
        \(X\) is Lindel\"{o}f.
        \item \label{BasicLambdaLindelof}
        \(X\) is \(\lambda\)-Lindel\"{o}f.
        \item \label{BasicOmegaLambdaLindelof}
        \(X \models \binom{\Omega}{\Lambda}\).
        \item \label{BasicOmegaOpen}
        \(X \models \binom{\Omega}{\mathcal O}\).
        \item \label{BasicKLambdaLindelof}
        \(X \models \binom{\mathcal K}{\Lambda}\).
        \item \label{BasicKOpen}
        \(X \models \binom{\mathcal K}{\mathcal O}\).
    \end{enumerate}
\end{lemma}
\begin{proof}
    We start by showing \ref{BasicLindelof}\(\implies\)\ref{BasicLambdaLindelof}.
    Suppose \(X\) is Lindel\"{o}f and that \(\mathscr U\) is a large cover of \(X\).
    Set \(\mathscr F_0 = \varnothing\) and choose \(\{ U_{0,n} : n \in \omega \} \subseteq \mathscr U\) which covers \(X\).
    Note that \(\{ U_{0,n} : n \in \omega \}\) need not necessarily be infinite.
    For \(k \in \omega\), suppose we've defined \(\left\{ \mathscr F_j \in \left[ \mathscr U \right]^{<\aleph_0} : j \leq k \right\}\) and
    \(\left\{\{ U_{j,n} : n \in \omega \} : j \leq k \right\}\).
    Then define
    \[\mathscr F_{k+1} = \{ U_{j,\ell} : j,\ell \leq k \}.\]
    Since \(\mathscr U\) is a large cover, \(\mathscr U \setminus \mathscr F_{k+1}\) is a cover of \(X\).
    So we can choose \[\{ U_{k+1,n} : n \in \omega \} \subseteq \mathscr U \setminus \mathscr F_{k+1}\] which covers \(X\)
    (and is not necessarily infinite).
    This defines, for all \(m \in \omega\), \(\{ U_{m,n} : n \in \omega \} \subseteq \mathscr U\).
    
    We claim that \(\{ U_{n,m} : n, m \in \omega \} \subseteq \mathscr U\) is a large cover of \(X\).
    So let \(x \in X\), set \(g(0) = 0\), and choose \(f(0) \in \omega\) for which
    \(x \in U_{g(0),f(0)}\).
    Now, for \(k \in \omega\), suppose we have defined
    \(\langle g(j) \in \omega : j \leq k \rangle\) and \(\langle f(j) \in \omega : j \leq k \rangle\) with the properties that
    \begin{itemize}[leftmargin=2em]
        \item \(\langle g(j) : j \leq k \rangle\) is strictly increasing,
        \item \(f(j) < g(j+1)\) for each \(j < k\), and
        \item \(x \in U_{g(j) , f(j)}\) for each \(j \leq k\).
    \end{itemize}
    Define \(g(k+1) = 1 + g(k) + \max \{ f(j) : j \leq k \}\) and note that
    \[\{ U_{g(j), f(j)} : j \leq k \} \subseteq \mathscr F_{g(k+1)}.\]
    Then there must be some \(f(k+1) \in \omega\)
    for which \[x \in U_{g(k+1), f(k+1)} \in \mathscr U \setminus \mathscr F_{g(k+1)}.\]
    This completes an injective enumeration \(\langle U_{g(n), f(n)} : n \in \omega \rangle\)
    of elements of \(\mathscr U\) which contain \(x\).
    Conclusively, \(\{ U_{n,m} : n,m \in \omega\}\) is a large cover of \(X\).

    The implications
    \ref{BasicLambdaLindelof}\(\implies\)\ref{BasicOmegaLambdaLindelof}
    and
    \ref{BasicOmegaLambdaLindelof}\(\implies\)\ref{BasicKLambdaLindelof}
    are immediate.
    It is also clear that \ref{BasicOmegaLambdaLindelof}\(\implies\)\ref{BasicOmegaOpen}
    and \ref{BasicKLambdaLindelof}\(\implies\)\ref{BasicKOpen}.

    We now show that
    \ref{BasicOmegaOpen}\(\implies\)\ref{BasicLindelof}
    and \ref{BasicKOpen}\(\implies\)\ref{BasicLindelof}
    simultaneously by way of the contrapositive.
    Suppose \(X\) fails to be Lindel\"{o}f and let \(\mathscr U\)
    be an open cover of \(X\) which has no countable subcover.
    In particular, note that \(\mathscr U\) must be infinite and has no finite subcover.
    Let
    \[\mathscr W = \left\{ \bigcup \mathscr F : \mathscr F \in \left[ \mathscr U \right]_+^{<\aleph_0} \right\}\]
    and note that \(\mathscr W\) is a \(k\)-cover of \(X\) and an \(\omega\)-cover of \(X\).
    We claim additionally that \(\mathscr W\) has no countable subcover.
    Indeed, note that, for any countable \(\mathscr F \subseteq \mathscr W\), there exists a countable \(\mathscr G \subseteq \mathscr U\)
    such that \(\bigcup \mathscr F = \bigcup \mathscr G\).
    Since \(\mathscr U\) admits no countable subcover, we see that \(\mathscr W\) also admits no countable
    subcover.
\end{proof}

\subsection{The Case for Finite-Selections} \label{subsec:FiniteSelections}

The primary objective of this section is to prove Theorem \ref{thm:LambdaTypeMengerRelations},
which is an equivalence of Menger-type games involving large covers, \(\omega\)-covers, and \(k\)-covers.
We emphasize that we include the cases for \(k\)-covers here for the sake of completion and as a way to contrast
against the situation with the analogous Rothberger-type games.
In particular, one of the main techniques in this section is closing an open cover under finite unions,
which either results in a trivial cover or in a \(k\)-cover.
This technique can clearly not be effectively carried out in the single-selection games, in general.

The following is an expansion of \cite[Theorem 5]{ScheepersPartitionRelations} with a consolidated proof provided for
the benefit of the reader, especially as we will refer to the proof of \ref{equivMenger}\(\implies\)\ref{equivStrategic}
in Lemma \ref{lem:RothbergerEquivalence}.
\begin{lemma} \label{lem:MengerEquivalence}
    For any space \(X\), the following are equivalent.
    \begin{enumerate}[label=(\roman*)]
        \item \label{equivMenger}
        \(X\) is Menger.
        \item \label{equivLambdaMenger}
        \(X\) is \(\lambda\)-Menger.
        \item \label{equivOmegaLambda}
        \(X \models \mathsf S_{\mathrm{fin}}(\Omega, \Lambda)\).
        \item \label{equivOmegaOpen}
        \(X \models \mathsf S_{\mathrm{fin}}(\Omega, \mathcal O)\).
        \item \label{equivKLambda}
        \(X \models \mathsf S_{\mathrm{fin}}(\mathcal K, \Lambda)\).
        \item \label{equivKOpen}
        \(X \models \mathsf S_{\mathrm{fin}}(\mathcal K, \mathcal O)\).
        \item \label{equivMengerStrategic}
        \(\mathrm{I} \not\uparrow \mathsf G_{\mathrm{fin}}(\mathcal O_X, \mathcal O_X)\).
        \item \label{equivStrategic}
        \(\mathrm{I} \not\uparrow \mathsf G_{\mathrm{fin}}(\Lambda_X, \Lambda_X)\).
        \item \label{equivStrategicExtra}
        \(\mathrm{I} \not\uparrow \mathsf G_{\mathrm{fin}}(\Omega_X, \Lambda_X)\).
        \item \label{equivStrategicExtraExtra}
        \(\mathrm{I} \not\uparrow \mathsf G_{\mathrm{fin}}(\Omega_X, \mathcal O_X)\).
        \item \label{equivStratKLambda}
        \(\mathrm{I} \not\uparrow \mathsf G_{\mathrm{fin}}(\mathcal K_X, \Lambda_X)\).
        \item \label{equivStratKOpen}
        \(\mathrm{I} \not\uparrow \mathsf G_{\mathrm{fin}}(\mathcal K_X, \mathcal O_X)\).
    \end{enumerate}
\end{lemma}
\begin{proof}
    The equivalence of \ref{equivMenger} and \ref{equivMengerStrategic} is the well-known theorem of Hurewicz.
    The implications
    \ref{equivLambdaMenger}\(\implies\)\ref{equivOmegaLambda},
    \ref{equivOmegaLambda}\(\implies\)\ref{equivOmegaOpen},
    \ref{equivOmegaLambda}\(\implies\)\ref{equivKLambda},
    \ref{equivOmegaOpen}\(\implies\)\ref{equivKOpen},
    \ref{equivKLambda}\(\implies\)\ref{equivKOpen},
    \ref{equivStrategic}\(\implies\)\ref{equivStrategicExtra},
    \ref{equivStrategicExtra}\(\implies\)\ref{equivStrategicExtraExtra},
    \ref{equivStrategicExtra}\(\implies\)\ref{equivStratKLambda},
    \ref{equivStrategicExtraExtra}\(\implies\)\ref{equivStratKOpen},
    \ref{equivStratKLambda}\(\implies\)\ref{equivStratKOpen}
    \ref{equivStrategic}\(\implies\)\ref{equivLambdaMenger},
    \ref{equivStrategicExtra}\(\implies\)\ref{equivOmegaLambda},
    \ref{equivStrategicExtraExtra}\(\implies\)\ref{equivOmegaOpen},
    \ref{equivStratKLambda}\(\implies\)\ref{equivKLambda},
    and \ref{equivStratKOpen}\(\implies\)\ref{equivKOpen}
    are immediate.
    
    We show that \ref{equivKOpen}\(\implies\)\ref{equivMenger} by the contrapositive.
    Assume \(X\) is a space which fails to be Menger and let \(\langle \mathscr U_n : n \in \omega \rangle\)
    be a sequence of open covers of \(X\) such that, for any
    \[\langle \mathscr F_n : n \in \omega \rangle \in \prod_{n\in\omega} \left[ \mathscr U_n \right]_+^{<\aleph_0},\]
    \(\bigcup\{ \mathscr F_n : n \in \omega \}\) fails to be a cover of \(X\).
    Note that, in particular, \(\mathscr U_n\) fails to have a finite subcover for each \(n \in \omega\).
    Then, consider the sequence \(\left\langle \widetilde{\mathscr U}_n : n \in \omega \right\rangle\) of
    \(k\)-covers of \(X\).
    Note that any sequence of finite selections from \(\left\langle \widetilde{\mathscr U}_n : n \in \omega \right\rangle\)
    corresponds to a sequence of finite selections \(\langle \mathscr U_n : n \in \omega \rangle\), which means that any such
    selection fails to cover \(X\).
    That is,
    \(X \not\models \mathsf S_{\mathrm{fin}}(\mathcal K, \mathcal O)\).

    To finish the proof, we show that
    \ref{equivMenger}\(\implies\)\ref{equivStrategic}
    using the suggested technique in \cite[Theorem 5]{ScheepersPartitionRelations}.
    Suppose that \(X\) is Menger and let \(\sigma\) be any strategy for P1
    in \(\mathsf G_{\mathrm{fin}}(\Lambda_X, \Lambda_X)\).
    Since \(X\) is Menger, \(X\) is Lindel\"{o}f and, hence, is \(\lambda\)-Lindel\"{o}f
    by Lemma \ref{lem:BasicLambdaLindelof}.
    Thus, we can fix a choice \(\gamma : \Lambda_X \to \mathscr T_X^\omega\) such that,
    for each \(\mathscr U \in \Lambda_X\),
    \[\{ \gamma(\mathscr U)_n : n \in \omega \} \in \wp(\mathscr U) \cap \Lambda_X\]
    and \(\gamma(\mathscr U) : \omega \to \mathscr U\) is injective.

    Since \(X\) is Menger, \(X \times \omega\) is Menger.
    We define an auxiliary strategy \(\widetilde\sigma\) in the game
    \(\mathsf G_{\mathrm{fin}}(\mathcal O_{X \times \omega} , \mathcal O_{X \times \omega})\)
    as follows.
    We first set
    \[\widetilde\sigma\langle \rangle = \left\{ \gamma(\sigma\langle \rangle)_m \times \{ n\} : m \geq n \geq 0 \right\}.\]
    Note that \(\widetilde\sigma\langle \rangle\) is an open cover of \(X \times \omega\)
    since \(\sigma\langle \rangle\) is a large cover of \(X\).

    For \(k \in \omega\), suppose we've defined \(\widetilde\sigma\) for every sequence
    \(\langle \mathscr F_j : j < k \rangle\) of legal plays by P2 in such a way that, if,
    given a sequence \(\langle \mathscr F_j : j < k \rangle\) of legal plays by P2,
    we have another sequence \(\langle \mathscr G_j : j < k \rangle\) of finite sets such that
    \(U \in \mathscr G_j\) if and only if there is some \(n \in\omega\) for which
    \(U \times \{ n \} \in \mathscr F_j\) for each \(j < k\) with the property that, for every \(\ell \leq k\),
    \begin{align*}
        &\widetilde\sigma\langle \mathscr F_j : j < \ell \rangle\\
        &= \left\{ \gamma(\sigma\langle \mathscr G_j : j < \ell \rangle)_m \times \{ n \} : m \geq n \geq 0
        \wedge \gamma(\sigma\langle \mathscr G_j : j < \ell \rangle)_m \not\in \bigcup \{ \mathscr G_j : j < \ell \} \right\}.
    \end{align*}
    For \(\mathscr F_k \in \widetilde\sigma\langle \mathscr F_j : j < k \rangle\),
    let \[\mathscr G_k = \{ U \in \mathscr T_X : (\exists n \in \omega)\ U \times \{ n \} \in \mathscr F_k \}.\]
    Note that the cardinality of \(\mathscr G_k\) is bounded by the cardinality of \(\mathscr F_k\) and that,
    for any \(U \in \mathscr G_k\), there is some \(m \in \omega\) such that
    \(U = \gamma(\sigma\langle \mathscr G_j : j < k \rangle)_m\).
    In particular,
    \[\mathscr G_k \in \left[ \sigma\langle \mathscr G_j : j < k \rangle \right]_+^{<\aleph_0}.\]
    Then \(\sigma\langle \mathscr G_j : j \leq k \rangle \in \Lambda_X\) and we can define
    \begin{align*}
        &\widetilde\sigma\langle \mathscr F_j : j \leq k \rangle\\
        &= \left\{ \gamma(\sigma\langle \mathscr G_j : j \leq k \rangle)_m \times \{ n \} : m \geq n \geq 0
        \wedge \gamma(\sigma\langle \mathscr G_j : j \leq k \rangle)_m \not\in \bigcup_{j=0}^k \mathscr G_j \right\}.
    \end{align*}
    To see that \(\widetilde\sigma\langle \mathscr F_j : j \leq k \rangle\) is an open cover of \(X \times \omega\),
    let \((x,n) \in X \times \omega\).
    Since \(\bigcup \{ \mathscr G_j : j \leq k \}\) is finite we can let \(m_0 \in \omega\) be such that
    \[\gamma(\sigma\langle \mathscr G_j : j \leq k \rangle)_m \not\in \bigcup \{ \mathscr G_j : j \leq k \}\]
    for each \(m \geq m_0\).
    Then, since \[\{ \gamma(\sigma\langle \mathscr G_j : j \leq k \rangle)_m : m \in \omega \} \in \Lambda_X,\]
    we can find \(m \geq \max \{ m_0 , n \}\) such that \(x \in \gamma(\sigma\langle \mathscr G_j : j \leq k \rangle)_m\).
    In particular,
    \[(x,n) \in \gamma(\sigma\langle \mathscr G_j : j \leq k \rangle)_m \times \{n\}.\]
    Thus, \(\widetilde\sigma\) is defined.

    Now, since \(X \times \omega\) is Menger, \(\widetilde\sigma\) cannot be a winning strategy by the Hurewicz Theorem.
    So there is some sequence \(\langle \mathscr F_k : k \in \omega \rangle\) such that,
    for each \(k \in \omega\),
    \[\mathscr F_k \in \left[ \widetilde\sigma\langle \mathscr F_j : j < k \rangle \right]^{<\aleph_0}_+\]
    and
    \[\bigcup \{ \mathscr F_k : k \in \omega\} \in \mathcal O_{X \times \omega}.\]
    Note that, by construction,
    \[\mathscr G_k = \left\{ U \in \mathscr T_X : (\exists n \in \omega)\ U \times \{n\} \in \mathscr F_k \right\}\]
    for each \(k \in \omega\) has the property that
    \[\mathscr G_k \in \left[ \sigma\langle \mathscr G_j : j < k \rangle\right]_+^{<\aleph_0}.\]
    We claim that
    \[\bigcup \{ \mathscr G_k : k \in \omega \} \in \Lambda_X.\]
    So let \(x \in X\) and let \(k \in \omega\) be arbitrary.
    Then let
    \[p = 1 + \max \left\{ n \in \omega : (\exists U \in \mathscr T_X)\ U \times \{ n \} \in \bigcup_{j=0}^k \mathscr F_j \right\},\]
    which is finite since \(\bigcup \{ \mathscr F_j : j \leq k \}\) is finite.
    Now, consider \((x,p) \in X \times \omega\).
    There must be \(m \in \omega\) and \(U \times \{ p \} \in \mathscr F_m\) such that
    \((x,p) \in U \times \{p\}\).
    By the definition of \(p\), \(U \not\in \bigcup_{j=0}^k \mathscr G_j\).
    Hence, \(m > k\) and \(x \in U \in \mathscr G_m\).
    Since \(k\) was arbitrary, \(\bigcup\{ \mathscr G_k : k \in \omega \}\) is a large cover of \(X\).
\end{proof}
\begin{lemma} \label{lem:TwoStrategic}
    For any space \(X\), the following are equivalent.
    \begin{enumerate}[label=(\roman*)]
        \item \label{TwoStratMenger}
        \(\mathrm{II} \uparrow \mathsf G_{\mathrm{fin}}(\mathcal O_X, \mathcal O_X)\).
        \item \label{TwoStratLambdaMenger}
        \(\mathrm{II} \uparrow \mathsf G_{\mathrm{fin}}(\Lambda_X, \Lambda_X)\).
        \item \label{TwoStratOmegaLambdaMenger}
        \(\mathrm{II} \uparrow \mathsf G_{\mathrm{fin}}(\Omega_X, \Lambda_X)\).
        \item \label{TwoStratOmegaLambdaMengerExtra}
        \(\mathrm{II} \uparrow \mathsf G_{\mathrm{fin}}(\Omega_X, \mathcal O_X)\).
        \item \label{TwoStratKLambda}
        \(\mathrm{II} \uparrow \mathsf G_{\mathrm{fin}}(\mathcal K_X, \Lambda_X)\).
        \item \label{TwoStratKO}
        \(\mathrm{II} \uparrow \mathsf G_{\mathrm{fin}}(\mathcal K_X, \mathcal O_X)\).
    \end{enumerate}
\end{lemma}
\begin{proof}
    First, note that the implications
    \ref{TwoStratLambdaMenger}\(\implies\)\ref{TwoStratOmegaLambdaMenger},
    \ref{TwoStratOmegaLambdaMenger}\(\implies\)\ref{TwoStratOmegaLambdaMengerExtra},
    \ref{TwoStratOmegaLambdaMenger}\(\implies\)\ref{TwoStratKLambda},
    \ref{TwoStratOmegaLambdaMengerExtra}\(\implies\)\ref{TwoStratKO},
    and \ref{TwoStratKLambda}\(\implies\)\ref{TwoStratKO}
    are immediate.
    
    We start by showing that \ref{TwoStratMenger}\(\implies\)\ref{TwoStratLambdaMenger}.
    The idea of the proof is to play the Menger game infinitely many times in such a way
    that one diagonalizes away from all initial segments of choices, which we formalize below.

    Fix a bijection \(\beta : \omega^2 \to \omega\) such that, for each \(n \in \omega\),
    \(\langle \beta(n,m) : m \in \omega \rangle\) is increasing.
    Also, for \(n \in \omega\), let \(( a_n, b_n ) = \beta^{-1}(n)\).
    Suppose \(\hat\tau\) is a winning strategy for P2 in \(\mathsf G_{\mathrm{fin}}(\mathcal O_X, \mathcal O_X)\).
    We define a strategy \(\tau\) for P2 in \(\mathsf G_{\mathrm{fin}}(\Lambda_X, \Lambda_X)\) as follows.
    For \(k \in \omega\), suppose we have defined \(\tau\) for every sequence
    \(\langle \mathscr U_j : j < k \rangle\) of large covers of \(X\).
    Additionally suppose that, for each \(j \leq k\) and each sequence
    \(\langle \mathscr U_\ell : \ell < j \rangle\) of large covers of \(X\),
    we have defined an auxiliary sequence \(\langle \mathscr V_\ell : \ell < j \rangle\)
    of open covers of \(X\).
    Then, given a sequence \(\langle \mathscr U_j : j \leq k \rangle\) of large covers of \(X\),
    let
    \[\mathscr V_k = \mathscr U_k \setminus \bigcup \left\{
    \bigcup \left\{ \hat\tau\langle \mathscr V_{\beta(a_j,t)} : t \leq m \rangle : m \leq b_j \right\} : j < k \right\}\]
    and define
    \[\tau\langle \mathscr U_j : j \leq k \rangle = \hat\tau\langle \mathscr V_{\beta(a_k,t)} :t \leq b_k \rangle.\]
    
    Now, we show that \(\tau\) is winning.
    Suppose \(\langle \mathscr U_n : n \in \omega \rangle\) is a sequence of large covers of \(X\)
    and that \(\langle \mathscr V_n : n \in \omega \rangle\) follows the auxiliary sequence
    construction defined above.
    Then consider \(x \in X\) and suppose \(n \in \omega\) is arbitrary.
    Let \[\mathscr F = \{ V \in \tau \langle \mathscr U_\ell : \ell \leq j \rangle : x \in V \wedge j \leq n \}.\]
    We can then let \(p \in \omega\) be such that \(n < \beta(a_p,0)\).
    Then there is some \(k \in \omega\) and \(W \in \hat\tau\langle \mathscr V_{\beta(a_p,j)} : j \leq k \rangle\)
    such that
    \[x \in W \in \hat\tau\langle \mathscr V_{\beta(a_p,j)} : j \leq k \rangle \subseteq \mathscr V_{\beta(a_p,k)}.\]
    Note that
    \[\beta(a_n, b_n) = n < \beta(a_p,0) \leq \beta(a_p,k) =: m = \beta(a_m, b_m)\]
    and that, hence,
    \[x \in W \in \hat\tau\langle \mathscr V_{\beta(a_m,t)} : t \leq b_m \rangle = \tau\langle \mathscr U_j : j \leq m \rangle.\]
    Moreover,
    \[W \in \mathscr V_m = \mathscr U_m \setminus \bigcup \left\{
    \bigcup \left\{ \hat\tau\langle \mathscr V_{\beta(a_j,t)} : t \leq \ell \rangle : \ell \leq b_j \right\} : j < m \right\}.\]
    Note that, in particular, for each \(j \leq n\),
    \[\tau\langle \mathscr U_\ell : \ell \leq j \rangle
    = \hat\tau \langle \mathscr V_{\beta(a_j, t)} : t \leq b_j \rangle\]
    and \(\beta(a_j, b_j) = j \leq n < m\).
    Hence,
    \begin{align*}
    \mathscr F \subseteq \bigcup_{j \leq n} \tau\langle \mathscr U_\ell : \ell \leq j \rangle
    &= \bigcup_{j \leq n} \hat\tau \langle \mathscr V_{\beta(a_j, t)} : t \leq b_j \rangle\\
    &\subseteq \bigcup_{j \leq n} \bigcup \left\{ \hat\tau \langle \mathscr V_{\beta(a_j, t)} : t \leq \ell \rangle : \ell \leq b_j \right\}\\
    &\subseteq \bigcup_{j < m} \bigcup \left\{ \hat\tau \langle \mathscr V_{\beta(a_j, t)} : t \leq \ell \rangle : \ell \leq b_j \right\}.
    \end{align*}
    It follows that \(W \not\in \mathscr F\).
    Therefore, since \(n \in \omega\) was arbitrary, we see that
    \[\bigcup \{ \tau \langle \mathscr U_j : j \leq n \rangle : n \in \omega \}\]
    is a large cover of \(X\).

    We finish by showing \ref{TwoStratKO}\(\implies\)\ref{TwoStratMenger}.
    Suppose P2 has a winning strategy \(\hat\tau\) in \(\mathsf{G}_{\mathrm{fin}}(\mathcal K_X, \mathcal O_X)\) and,
    given any family \(\mathscr U\) of subsets of \(X\), let
    \[\widetilde{\mathscr U} = \left\{ \bigcup \mathscr F : \mathscr F \in \left[ \mathscr U \right]_+^{<\aleph_0} \right\}\]
    and fix a choice \[\mathscr G_{\mathscr U} : \widetilde{\mathscr U} \to \left[ \mathscr U \right]_+^{<\aleph_0}\]
    such that, for each \(W \in \widetilde{\mathscr U}\),
    \[W = \bigcup \mathscr G_{\mathscr U}(W).\]
    We define \(\tau\) in \(\mathsf G_{\mathrm{fin}}(\mathcal O_X, \mathcal O_X)\) in the following manner.
    Let \(n \in \omega\) and suppose \(\tau\) has been defined for every \(\langle \mathscr U_j : j < n \rangle\)
    where \(\mathscr U_j \in \mathcal O_X\) for each \(j < n\).
    Suppose \(\langle \mathscr U_j : j \leq n \rangle\) is a sequence of open covers of \(X\).

    \begin{enumerate}
        \item \label{LambdaTheoremFiniteSubcover}
        If there is some \(\mathscr V \in \left[ \mathscr U_n \right]_+^{<\aleph_0}\) such that \(X = \bigcup \mathscr V\),
        then define \[\tau\langle \mathscr U_j : j \leq n \rangle = \mathscr V.\]
        \item 
        If \(\left\langle \widetilde{\mathscr U}_j : j \leq n \right\rangle\) is a sequence of \(k\)-covers of \(X\),
        then let
        \[\tau\langle \mathscr U_j : j \leq n \rangle =
        \bigcup\mathscr G_{\mathscr U_n}\left[ \hat\tau\left\langle \widetilde{\mathscr U}_j : j \leq n \right\rangle \right].\]
        \item 
        Otherwise, there is some \(j < n\) for which \(\widetilde{\mathscr U}_j\) is not a \(k\)-cover of \(X\).
        In particular, by our conditions, \(\widetilde{\mathscr U}_j\) has a finite subcover, and,
        by \eqref{LambdaTheoremFiniteSubcover},
        \[X = \bigcup\tau\langle \mathscr U_k : k \leq j \rangle.\]
        So we can choose
        \[\tau\langle \mathscr U_j : j \leq n \rangle \in \left[ \mathscr U_n \right]_+^{<\aleph_0}\]
        arbitrarily.
    \end{enumerate}

    To see that \(\tau\) is winning, consider any sequence \(\langle \mathscr U_n : n \in \omega \rangle\)
    of open covers of \(X\).
    If there is an \(n \in \omega\) for which \(\mathscr U_n\) admits a finite subcover, then we see that
    \[X = \bigcup \tau \langle \mathscr U_j : j \leq n \rangle
    \subseteq \bigcup \left\{ \bigcup\tau\langle \mathscr U_j : j \leq k \rangle : k \in \omega \right\}.\]
    Otherwise,
    \[X = \bigcup \left\{ \bigcup\hat\tau\left\langle \widetilde{\mathscr U}_j : j \leq k \right\rangle : k \in \omega \right\}
    = \bigcup \left\{ \bigcup \tau\langle \mathscr U_j : j \leq k \rangle : k \in \omega \right\},\]
    finishing the proof.
\end{proof}
\begin{lemma} \label{lem:TwoMarkovMenger}
    For any space \(X\), the following are equivalent.
    \begin{enumerate}[label=(\roman*)]
        \item \label{TwoMarkovMenger}
        \(\mathrm{II} \underset{\mathrm{mark}}{\uparrow} \mathsf G_{\mathrm{fin}}(\mathcal O_X, \mathcal O_X)\).
        \item \label{TwoMarkovLambdaMenger}
        \(\mathrm{II} \underset{\mathrm{mark}}{\uparrow} \mathsf G_{\mathrm{fin}}(\Lambda_X, \Lambda_X)\).
        \item \label{TwoMarkovOmegaLambdaMenger}
        \(\mathrm{II} \underset{\mathrm{mark}}{\uparrow} \mathsf G_{\mathrm{fin}}(\Omega_X, \Lambda_X)\).
        \item \label{TwoMarkovOmegaLambdaMengerExtra}
        \(\mathrm{II} \underset{\mathrm{mark}}{\uparrow} \mathsf G_{\mathrm{fin}}(\Omega_X, \mathcal O_X)\).
        \item \label{TwoMarkovKLambdaMenger}
        \(\mathrm{II} \underset{\mathrm{mark}}{\uparrow} \mathsf G_{\mathrm{fin}}(\mathcal K_X, \Lambda_X)\).
        \item \label{TwoMarkovKOMenger}
        \(\mathrm{II} \underset{\mathrm{mark}}{\uparrow} \mathsf G_{\mathrm{fin}}(\mathcal K_X, \mathcal O_X)\).
    \end{enumerate}
\end{lemma}
\begin{proof}
    First, note that the implications
    \ref{TwoMarkovLambdaMenger}\(\implies\)\ref{TwoMarkovOmegaLambdaMenger},
    \ref{TwoMarkovOmegaLambdaMenger}\(\implies\)\ref{TwoMarkovOmegaLambdaMengerExtra},
    \ref{TwoMarkovOmegaLambdaMenger}\(\implies\)\ref{TwoMarkovKLambdaMenger},
    \ref{TwoMarkovOmegaLambdaMengerExtra}\(\implies\)\ref{TwoMarkovKOMenger},
    and \ref{TwoMarkovKLambdaMenger}\(\implies\)\ref{TwoMarkovKOMenger}
    are immediate.

    We prove that \ref{TwoMarkovKOMenger}\(\implies\)\ref{TwoMarkovMenger}.
    Let \(\hat\tau\) be a winning Markov strategy for P2 in \(\mathsf{G}_{\mathrm{fin}}(\mathcal K_X, \mathcal O_X)\) and,
    given any family \(\mathscr U\) of subsets of \(X\), let
    \[\widetilde{\mathscr U} = \left\{ \bigcup \mathscr F : \mathscr F \in \left[ \mathscr U \right]_+^{<\aleph_0} \right\}.\]
    Fix a choice \[\mathscr G_{\mathscr U} : \widetilde{\mathscr U} \to \left[ \mathscr U \right]_+^{<\aleph_0}\]
    such that, for each \(W \in \widetilde{\mathscr U}\),
    \[W = \bigcup \mathscr G_{\mathscr U}(W).\]
    For \(\mathscr U \in \mathcal O_X\) and \(n \in \omega\), define \(\tau\) according to the following.

    If there is some \(\mathscr V \in \left[ \mathscr U \right]_+^{<\aleph_0}\) such that
    \(X = \bigcup \mathscr V\), let \(\tau(\mathscr U, n) = \mathscr V\).
    Otherwise, \(\mathscr U\) has no finite subcover of \(X\).
    Then let \[\tau(\mathscr U, n) = \bigcup \mathscr G_{\mathscr U}\left[ \hat\tau\left( \widetilde{\mathscr U}, n\right)\right].\]

    To see that \(\tau\) is winning, consider any sequence \(\langle \mathscr U_n : n \in \omega \rangle\)
    of open covers of \(X\).
    If there is an \(n \in \omega\) for which \(\mathscr U_n\) admits a finite subcover, then we see that
    \[X = \bigcup \tau(\mathscr U_n, n ) \subseteq \bigcup \left\{ \bigcup \tau(\mathscr U_k, k) : k \in \omega \right\}.\]
    Otherwise,
    \[X = \bigcup \left\{ \bigcup\hat\tau\left( \widetilde{\mathscr U}_k, k \right) : k \in \omega \right\}
    = \bigcup \left\{ \bigcup \tau( \mathscr U_k , k ) : k \in \omega \right\}.\]

    To finish the proof, we show that
    \ref{TwoMarkovMenger}\(\implies\)\ref{TwoMarkovLambdaMenger}.
    By \cite[Theorem 4.17]{CCHMengerRothbergerSurvey}, \(X\) is Markov Menger if and only if \(X\) is
    \(\sigma\)-relatively compact.
    So we prove here that every \(\sigma\)-relatively compact space is Markov \(\lambda\)-Menger.

    For a space \(X\), a relatively compact \(A \subseteq X\), a large cover \(\mathscr U\) of \(X\), and
    an integer \(n \in \omega\), we define a choice function as follows:
    choose
    \[\widetilde{\mathscr F}(\mathscr U, A, n) \in \left[ \mathscr U \setminus
    \bigcup \{ \widetilde{\mathscr F}(\mathscr U, A, j) : j < n \} \right]_+^{<\aleph_0}\]
    with the property that
    \[A \subseteq \bigcup \widetilde{\mathscr F}(\mathscr U, A, n).\]
    Then set
    \[\mathscr F(\mathscr U, A, n) = \bigcup \left\{ \widetilde{\mathscr F}(\mathscr U, A, j) : j \leq n \right\}.\]
    Note that, for any \(x \in A\), there are at least \(n\)-many elements of \(\mathscr F(\mathscr U, A, n)\)
    which contain \(x\).
    
    Now suppose \(X\) is \(\sigma\)-relatively compact and let \(\{ A_n : n \in \omega \}\)
    be relatively compact subsets of \(X\) with \(X = \bigcup \{ A_n : n \in \omega \}\).
    Also fix a bijection \(\beta : \omega \to \omega^2\) and write \(\beta(n) = ( \beta(n)_1, \beta(n)_2 )\).
    Define \(\tau : \Lambda_X \times \omega \to \left[ \mathscr T_X \right]_+^{<\aleph_0}\) by
    \[\tau(\mathscr U,n) = \mathscr F(\mathscr U, A_{\beta(n)_1}, \beta(n)_2).\]
    To see that this is a winning Markov strategy for P2 in \(\mathsf G_{\mathrm{fin}}(\Lambda_X, \Lambda_X)\),
    let \(\langle \mathscr U_n : n \in \omega \rangle\) be a sequence of large covers of \(X\).
    Then consider an arbitrary \(x \in X\) and \(m \in \omega\).
    Note that there must be some \(n \in \omega\) such that
    \(x \in A_{\beta(n)_1}\) and \(m < \beta(n)_2\).
    It follows that there are at least \(\beta(n)_2 > m\) distinct elements of
    \(\tau(\mathscr U_n, n)\) which contain \(x\).
    Since \(m \in \omega\) was arbitrary, we see that there must be infinitely many members of
    \(\bigcup \{ \tau(\mathscr U_n, n ) : n \in \omega \}\) which contain \(x\),
    finishing the proof.
\end{proof}

\begin{theorem} \label{thm:LambdaTypeMengerRelations}
    For any space \(X\),
    \begin{align*}
        \mathsf G_{\mathrm{fin}}(\mathcal O_X, \mathcal O_X)
        &\leftrightarrows \mathsf G_{\mathrm{fin}}(\Lambda_X, \Lambda_X)
        \leftrightarrows \mathsf G_{\mathrm{fin}}(\Omega_X, \Lambda_X)
        \leftrightarrows \mathsf G_{\mathrm{fin}}(\Omega_X, \mathcal O_X)\\
        &\leftrightarrows \mathsf G_{\mathrm{fin}}(\mathcal K_X, \Lambda_X)
        \leftrightarrows \mathsf G_{\mathrm{fin}}(\mathcal K_X, \mathcal O_X).
    \end{align*}
\end{theorem}
\begin{proof}
    This is a combination of Lemmas \ref{lem:BasicLambdaLindelof}, \ref{lem:MengerEquivalence}, \ref{lem:TwoStrategic},
    and \ref{lem:TwoMarkovMenger}.
\end{proof}

\subsection{The Case for Single-Selections}

The primary objective of this section is to prove Theorem \ref{thm:SingleLambdaOpen}.
As mentioned at the beginning of Section \ref{subsec:FiniteSelections},
we cannot include \(k\)-covers in Theorem \ref{thm:SingleLambdaOpen} in the way they
appear in Theorem \ref{thm:LambdaTypeMengerRelations}.
As a particular example, note that \(\mathbb R\) fails to be Rothberger, but does
satisfy \(\mathsf S_1(\mathcal K, \mathcal O)\) since it is \(\sigma\)-compact.
Hence, we cannot include \(X \models \mathsf S_1(\mathcal K, \mathcal O)\)
as an equivalent condition in Lemma \ref{lem:RothbergerEquivalence}.

The following extends \cite[Theorem 3]{ScheepersPartitionRelations}
and \cite[Theorem 17]{COOC1}.
\begin{lemma} \label{lem:RothbergerEquivalence}
    For any space \(X\), the following are equivalent.
    \begin{enumerate}[label=(\roman*)]
        \item \label{equivRoth}
        \(X\) is Rothberger.
        \item \label{equivLambdaRoth}
        \(X\) is \(\lambda\)-Rothberger.
        \item \label{equivRothOmegaLambda}
        \(X \models \mathsf S_1(\Omega, \Lambda)\).
        \item \label{equivRothOmegaOpen}
        \(X \models \mathsf S_1(\Omega, \mathcal O)\).
        \item \label{equivRothStrategic}
        \(\mathrm{I} \not\uparrow \mathsf G_1(\mathcal O_X, \mathcal O_X)\).
        \item \label{equivStrategicLambdaRoth}
        \(\mathrm{I} \not\uparrow \mathsf G_1(\Lambda_X, \Lambda_X)\).
        \item \label{equivStrategicOmegaLambda}
        \(\mathrm{I} \not\uparrow \mathsf G_1(\Omega_X, \Lambda_X)\).
        \item \label{equivStrategicOmegaOpen}
        \(\mathrm{I} \not\uparrow \mathsf G_1(\Omega_X, \mathcal O_X)\).
    \end{enumerate}
\end{lemma}
\begin{proof}
    The equivalence of \ref{equivRoth} and \ref{equivRothStrategic} is the well-known theorem of Pawlikowski.
    Note that the implications
    \ref{equivLambdaRoth}\(\implies\)\ref{equivRothOmegaLambda},
    \ref{equivRothOmegaLambda}\(\implies\)\ref{equivRothOmegaOpen},
    \ref{equivStrategicLambdaRoth}\(\implies\)\ref{equivStrategicOmegaLambda},
    \ref{equivStrategicOmegaLambda}\(\implies\)\ref{equivStrategicOmegaOpen},
    \ref{equivStrategicLambdaRoth}\(\implies\)\ref{equivLambdaRoth},
    \ref{equivStrategicOmegaLambda}\(\implies\)\ref{equivRothOmegaLambda},
    and \ref{equivStrategicOmegaOpen}\(\implies\)\ref{equivRothOmegaOpen}
    are immediate.

    The implication \ref{equivRothOmegaOpen}\(\implies\)\ref{equivRoth}
    can be proved in exactly the same way as it is proved in \cite[Theorem 17]{COOC1},
    without needing the stated hypothesis there that \(X\) is a set of real numbers.

    Finally, the proof is complete by noting that \ref{equivRoth}\(\implies\)\ref{equivStrategicLambdaRoth}
    follows from the proof of \ref{equivMenger}\(\implies\)\ref{equivStrategic} in
    Lemma \ref{lem:MengerEquivalence} by observing that, whenever \(\mathscr F_j\) is a singleton,
    \(\mathscr G_j\) is a singleton, and by replacing the use of the Hurewicz Theorem with the use
    of the Pawlikowski Theorem.
\end{proof}

We will now turn our attention to the second player's perspective and unravel
the differences in the single-selection context as compared to the finite-selection context.
In particular, Theorem \ref{thm:CountableNotMarkoveLambdaRoth} lies in direct opposition to
the equivalence reflected in Theorem \ref{thm:LambdaTypeMengerRelations} for finite-selection games.

The proof of the following is quite similar to the proof of Lemma \ref{lem:TwoStrategic},
except for, notably, the implication \ref{TwoSingleStratOmegaLambdaMenger}\(\implies\)\ref{TwoSingleStratMenger},
which passes through dual games in this setting as compared to the more direct approach
in the finite-selections setting.
\begin{lemma} \label{lem:TwoSingleStrategic}
    For any space \(X\), the following are equivalent.
    \begin{enumerate}[label=(\roman*)]
        \item \label{TwoSingleStratMenger}
        \(\mathrm{II} \uparrow \mathsf G_1(\mathcal O_X, \mathcal O_X)\).
        \item \label{TwoSingleStratLambdaMenger}
        \(\mathrm{II} \uparrow \mathsf G_1(\Lambda_X, \Lambda_X)\).
        \item \label{TwoSingleStratOmegaLambdaMenger}
        \(\mathrm{II} \uparrow \mathsf G_1(\Omega_X, \Lambda_X)\).
        \item \label{TwoSingleStratOmegaLambdaMengerExtra}
        \(\mathrm{II} \uparrow \mathsf G_1(\Omega_X, \mathcal O_X)\).
    \end{enumerate}
\end{lemma}
\begin{proof}
    We start by showing that \ref{TwoSingleStratMenger}\(\implies\)\ref{TwoSingleStratLambdaMenger},
    which follows the proof of the analogous implication in Lemma \ref{lem:TwoStrategic}.
    Fix a bijection \(\beta : \omega^2 \to \omega\) such that, for each \(n \in \omega\),
    \(\langle \beta(n,m) : m \in \omega \rangle\) is increasing.
    Also, for \(n \in \omega\), let \(( a_n, b_n ) = \beta^{-1}(n)\).
    Suppose \(\hat\tau\) is a winning strategy for P2 in \(\mathsf G_1(\mathcal O_X, \mathcal O_X)\).
    We define a strategy \(\tau\) for P2 in \(\mathsf G_1(\Lambda_X, \Lambda_X)\) as follows.
    For \(k \in \omega\), suppose we have defined \(\tau\) for every sequence
    \(\langle \mathscr U_j : j < k \rangle\) of large covers of \(X\).
    Additionally suppose that, for each \(j \leq k\) and each sequence
    \(\langle \mathscr U_\ell : \ell < j \rangle\) of large covers of \(X\),
    we have defined an auxiliary sequence \(\langle \mathscr V_\ell : \ell < j \rangle\)
    of open covers of \(X\).
    Then, given a sequence \(\langle \mathscr U_j : j \leq k \rangle\) of large covers of \(X\),
    let
    \[\mathscr V_k = \mathscr U_k \setminus \left\{ \hat\tau\langle \mathscr V_{\beta(a_j,t)} : t \leq m \rangle : j < k \wedge m \leq b_j \right\}\]
    and define
    \[\tau\langle \mathscr U_j : j \leq k \rangle = \hat\tau\langle \mathscr V_{\beta(a_k,t)} :t \leq b_k \rangle.\]
    
    Now, we show that \(\tau\) is winning.
    Suppose \(\langle \mathscr U_n : n \in \omega \rangle\) is a sequence of large covers of \(X\)
    and that \(\langle \mathscr V_n : n \in \omega \rangle\) follows the auxiliary sequence
    construction defined above.
    Then consider \(x \in X\) and suppose \(n \in \omega\) is arbitrary.
    Let \[\mathscr F = \{ \tau \langle \mathscr U_\ell
    : \ell \leq j \rangle : x \in \tau \langle \mathscr U_\ell : \ell \leq j \rangle \wedge j \leq n \}.\]
    We can then let \(p \in \omega\) be such that \(n < \beta(a_p,0)\).
    Then there is some \(k \in \omega\)
    such that
    \[x \in \hat\tau\langle \mathscr V_{\beta(a_p,j)} : j \leq k \rangle \in \mathscr V_{\beta(a_p,k)}.\]
    Note that
    \[\beta(a_n, b_n) = n < \beta(a_p,0) \leq \beta(a_p,k) =: m = \beta(a_m, b_m)\]
    and, thus, that
    \[x \in \hat\tau\langle \mathscr V_{\beta(a_m,t)} : t \leq b_m \rangle = \tau\langle \mathscr U_j : j \leq m \rangle.\]
    Thus,
    \[\hat\tau\langle \mathscr V_{\beta(a_p,j)} : j \leq k \rangle \in \mathscr V_m = \mathscr U_m \setminus 
    \left\{ \hat\tau\langle \mathscr V_{\beta(a_j,t)} : t \leq \ell \rangle : j < m \wedge \ell \leq b_j \right\}.\]
    Note that, in particular, for each \(j \leq n\),
    \[\tau\langle \mathscr U_\ell : \ell \leq j \rangle
    = \hat\tau \langle \mathscr V_{\beta(a_j, t)} : t \leq b_j \rangle\]
    and \(\beta(a_j, b_j) = j \leq n < m\).
    Hence,
    \begin{align*}
    \mathscr F \subseteq \{\tau\langle \mathscr U_\ell : \ell \leq j \rangle : j \leq n \}
    &= \{\hat\tau \langle \mathscr V_{\beta(a_j, t)} : t \leq b_j \rangle : j \leq n \}\\
    &\subseteq \bigcup_{j \leq n} \left\{ \hat\tau \langle \mathscr V_{\beta(a_j, t)} : t \leq \ell \rangle : \ell \leq b_j \right\}\\
    &\subseteq \left\{ \hat\tau \langle \mathscr V_{\beta(a_j, t)} : t \leq \ell \rangle : j < m \wedge \ell \leq b_j \right\}.
    \end{align*}
    It follows that \(\tau\langle \mathscr U_j : j \leq m \rangle \not\in \mathscr F\).
    Therefore, since \(n \in \omega\) was arbitrary, we see that
    \[\{ \tau \langle \mathscr U_j : j \leq n \rangle : n \in \omega \}\]
    is a large cover of \(X\).

    The implications \ref{TwoSingleStratLambdaMenger}\(\implies\)\ref{TwoSingleStratOmegaLambdaMenger}
    and \ref{TwoSingleStratOmegaLambdaMenger}\(\implies\)\ref{TwoSingleStratOmegaLambdaMengerExtra} are immediate.
    So we finish by proving \ref{TwoSingleStratOmegaLambdaMengerExtra}\(\implies\)\ref{TwoSingleStratMenger}.
    First note that
    \[\mathrm{II} \uparrow \mathsf G_1(\Omega_X, \mathcal O_X)
    \iff \mathrm{I} \uparrow \mathsf G_1(\mathcal N[X_{\mathrm{fin}}], \neg \mathcal O_X)\]
    by duality.
    Then, by the standard unraveling argument (see \cite[Corollary 4.3]{Telgarsky1975} and \cite[Lemma 4.1]{CCDimension}),
    \[\mathrm{I} \uparrow \mathsf G_1(\mathcal N[X_{\mathrm{fin}}], \neg \mathcal O_X)
    \iff \mathrm{I} \uparrow \mathsf G_1(\mathcal N[X], \neg \mathcal O_X).\]
    Again, by duality,
    \[\mathrm{I} \uparrow \mathsf G_1(\mathcal N[X], \neg \mathcal O_X)
    \iff \mathrm{II} \uparrow \mathsf G_1(\mathcal O_X, \mathcal O_X),\]
    and, thus, we have that
    \[\mathrm{II} \uparrow \mathsf G_1(\Omega_X, \mathcal O_X)
    \implies \mathrm{II} \uparrow \mathsf G_1(\mathcal O_X, \mathcal O_X),\]
    finishing the proof.
\end{proof}

\begin{proposition} \label{prop:MarkLambdaRothbergerIsTopCount}
    Let \(X\) be a space.
    If P2 has a winning Markov strategy in \(\mathsf G_1(\Omega_X, \mathcal O_X)\),
    then \(X\) is topologically countable.
    Consequently, if \(X\) is Markov \(\lambda\)-Rothberger,
    then \(X\) is topologically countable.
\end{proposition}
\begin{proof}
    If \(\Omega_X = \varnothing\), then \(X\) is topologically finite by
    Proposition \ref{prop:SillyNoOmegaThing}.
    So we suppose that \(\Omega_X \neq \varnothing\) and
    consider \(\tau\), a (not necessarily winning)
    Markov strategy for P2 in \(\mathsf G_1(\Omega_X, \mathcal O_X)\).
    Then, for each \(n \in\omega\), by \cite[Lemma 4.4]{CCHMengerRothbergerSurvey},
    we can let
    \(F_n \in [X]_+^{<\aleph_0}\) be such that
    \[\mathcal N_{F_n} \subseteq \{ \tau(\mathscr U, n) : \mathscr U \in \Omega_X \}.\]
    Consequently, with an application of Lemma \ref{lem:SillyNhoodThing}, we have that
    \[E_n := \bigcap \{ \tau(\mathscr U, n) : \mathscr U \in \Omega_X \} \subseteq \bigcap \mathcal N_{F_n}
    = \bigcup \left\{ \bigcap \mathcal N_x : x \in F_n \right\}.\]

    Now, by way of contrapositive, suppose there is some \(y \not\in \bigcup\{E_n:n\in\omega\}\).
    Then, for \(n\in\omega\), there must be some \(\mathscr U_n \in \Omega_X\) such that
    \(y \not\in \tau(\mathscr U_n,n)\).
    Hence, \(\langle \mathscr U_n : n \in \omega\rangle\) is a sequence of \(\omega\)-covers of \(X\)
    that beats \(\tau\).

    For the additional claim, note that
    \[
    \mathrm{II} \underset{\mathrm{mark}}{\uparrow} \mathsf G_1(\Lambda_X, \Lambda_X)
    \implies \mathrm{II} \underset{\mathrm{mark}}{\uparrow} \mathsf G_1(\Omega_X, \Lambda_X)
    \implies \mathrm{II} \underset{\mathrm{mark}}{\uparrow} \mathsf G_1(\Omega_X, \mathcal O_X),
    \]
    completing the proof.
\end{proof}
Now, unlike the contexts of \(\mathsf G_1(\Omega_X, \Omega_X)\) and \(\mathsf G_1(\mathcal O_X, \mathcal O_X)\),
not every topologically countable space is Markov \(\lambda\)-Rothberger.
\begin{theorem} \label{thm:CountableNotMarkoveLambdaRoth}
    If \(X\) is a countably infinite \(T_1\) space with an uncountable topology,
    then \(X\) is not Markov \(\lambda\)-Rothberger.
\end{theorem}
\begin{proof}
    Assuming that \(X\) is countable with an uncountable topology,
    there must be some point \(p \in X\) such that
    \(\mathcal N_p\) is uncountable.

    Now consider any function \(\tau : \Lambda_X \times \omega \to \mathscr T_X\)
    with the property that \(\tau(\mathscr U , n) \in \mathscr U\) for every \(\mathscr U \in \Lambda_X\)
    and \(n\in\omega\).
    For \(n \in \omega\), define \(A_n = \bigcap \{ \tau(\mathscr U, n) : \mathscr U \in \Lambda_X \}\).

    If \(M_p := \{ n \in \omega : p \in A_n \}\) is finite, for each \(n \in \omega \setminus M_p\),
    choose \(\mathscr U_n \in \Lambda_X\) for which \(p \not\in \tau(\mathscr U_n, n)\).
    For \(n \in M_p\), let \(\mathscr U_n \in \Lambda_X\) be arbitrary and note that
    \(\{ \tau(\mathscr U_n , n ) : n \in \omega \}\) fails to be a large cover since
    \(p\) is only covered by those \(\tau(\mathscr U_n, n)\) with \(n \in M_p\), which is finite.
    In this case, \(\tau\) is already seen to not be a winning strategy for P2 in
    \(\mathsf G_1(\Lambda_X, \Lambda_X)\).

    So we assume that \(M_p\) is infinite.
    For each \(n \in M_p\), consider
    \[\mathcal E_n = \{ \tau(\mathscr U, n) : \mathscr U \in \Lambda_X\}.\]
    Note that \(p \in \bigcap \mathcal E_n = A_n\) for \(n \in M_p\).
    We claim that \(\mathcal E_n \subseteq \mathcal N_p\) is cofinite in \(\mathcal N_p\).
    If it were the case that \(\mathcal N_p = \mathcal E_n\), then there would be clearly nothing to show.
    Otherwise, \(\mathcal N_p \setminus \mathcal E_n \neq \varnothing\).
    Then, for \(x \in X \setminus \{p\}\), let
    \[\mathcal F_x = \{ X \setminus \{p,y\} : y \in X \setminus \{x\} \}.\]
    Note that \(\mathcal F_x\) is an infinite family of open subsets of \(X\)
    that contain \(x\) and that \(p \not\in \bigcup \mathcal F_x\).
    We can then define
    \[\mathscr U = (\mathcal N_p \setminus \mathcal E_n) \cup \bigcup \{ \mathcal F_x : x \in X \setminus \{p\} \} \in \mathcal O_X.\]
    Since, for any \(\mathscr V \in \Lambda_X\), \(p \in \tau(\mathscr V, n) \in \mathcal E_n\) and,
    for any \(U \in \mathscr U\) with \(p \in U\), \(U \not\in \mathcal E_n\),
    we see that \(\mathscr U \not\in \Lambda_X\).
    Consequently, since every \(x \in X \setminus \{ p \}\) is covered infinitely often and
    \(\{ U \in \mathscr U : p \in U \} = \mathcal N_p \setminus \mathcal E_n\),
    we must have that \(\mathcal N_p \setminus \mathcal E_n\) is finite.

    Observe that \(\bigcup \{ \mathcal N_p \setminus \mathcal E_n : n \in M_p \}\) is countable and that
    \(\mathcal N_p\) is uncountable.
    So we can choose some
    \[U_p \in \mathcal N_p \setminus \bigcup \{ \mathcal N_p \setminus \mathcal E_n : n \in M_p \} = \bigcap \{ \mathcal E_n : n \in M_p \}.\]
    For each \(n \in M_p\), let \(\mathscr U_n \in \Lambda_X\) be such that \(U_p = \tau(\mathscr U_n, n)\).
    For \(n \in \omega \setminus M_p\), let \(\mathscr U_n \in \Lambda_X\) be such that
    \(p \not\in \tau(\mathscr U_n, n)\).
    Note then that
    \[\mathscr V := \{ \tau(\mathscr U_n, n) : n \in \omega \}\] fails to be a large cover of \(X\)
    since \(\{ V \in \mathscr V : p \in V \} = \{ U_p \}\).
    Therefore, \(\tau\) is not a winning strategy for P2 in
    \(\mathsf G_1(\Lambda_X, \Lambda_X)\).
\end{proof}
\begin{corollary} \label{cor:NontrivialSpacesAreNotLambdaRoth}
    If \(X\) is a countably infinite Hausdorff space, then \(X\) is not Markov
    \(\lambda\)-Rothberger.
\end{corollary}
\begin{proof}
    By Theorem \ref{thm:CountableNotMarkoveLambdaRoth} we need only show that \(X\)
    has an uncountable topology.
    So, let \(\{ V_n : n \in \omega \}\) be a set of pairwise disjoint open subsets of \(X\)
    and note that the mapping \(2^\omega \to \mathscr T_X \cup \{ \varnothing, X \}\) defined by
    \[\langle \alpha_n : n \in \omega \rangle \mapsto \bigcup \{ V_n : \alpha_n = 1 \}\]
    is an injection.
    Hence, \(X\) has an uncountable topology and we are done.
\end{proof}
\begin{example}
    The countable discrete space \(\omega\) is not Markov \(\lambda\)-Rothberger.
\end{example}
Despite this, there are still examples of infinite spaces which are Markov \(\lambda\)-Rothberger.
\begin{proposition} \label{prop:VerySmallMeansLambdaRoth}
    If \(X\) is a topologically countable space with a countable topology,
    then \(X\) is Markov \(\lambda\)-Rothberger.
\end{proposition}
\begin{proof}
    First, note that, if either \(\mathscr T_X\) is finite or \(\Lambda_X = \varnothing\),
    then \(X\) is vacuously Markov \(\lambda\)-Rothberger.
    So assume that \(\mathscr T_X\) is countably infinite and that \(\Lambda_X \neq \varnothing\).
    Then let \(\{ x_n : n \in \omega \} \subseteq X\) be such that
    \(X = \bigcup \left\{ \bigcap \mathcal N_{x_n} : n \in \omega \right\}\).
    Then fix bijections \(\varphi : \omega \to \mathscr T_X\) and \(\beta : \omega^2 \to \omega\).
    Without loss of generality, we can assume that, for each \(n \in \omega\),
    \(\langle \beta(n,k) : k \in \omega \rangle\) is strictly increasing.
    We also write \(( a_n, b_n ) = \beta^{-1}(n)\).
    We define \(\tau : \Lambda_X \times \omega \to \mathscr T_X\) to be a choice
    \(\tau(\mathscr U, n) \in \mathscr U\) such that
    \(x_{a_n} \in \tau(\mathscr U, n)\) and \(\varphi^{-1} \circ \tau(\mathscr U, n) \geq b_n\).
    Now, suppose \(\langle \mathscr U_n : n \in \omega \rangle\) is a sequence of large covers
    of \(X\) and consider \(x \in X\).
    Let \(m \in \omega\) be such that \(x \in \bigcap \mathcal N_{x_m}\).
    Then, for every \(k \in \omega\), note that
    \[x \in \bigcap \mathcal N_{x_m} \subseteq \tau(\mathscr U_{\beta(m,k)} , \beta(m,k)).\]
    Since \(\varphi^{-1} \circ \tau(\mathscr U_{\beta(m,k)} , \beta(m,k)) \geq k\), we see that
    \[\langle \tau(\mathscr U_{\beta(m,k)} , \beta(m,k)) : k \in \omega \rangle\]
    is unbounded in \(\langle \varphi(n) : n \in \omega \rangle\).
    Hence, \(x\) must be contained in infinitely many members of
    \(\{ \tau(\mathscr U_n, n ) : n \in \omega \}\).
\end{proof}
\begin{corollary} \label{cor:MarkovLambdaRothberger}
    If \(X\) is a \(T_1\) space, then \(X\) is Markov \(\lambda\)-Rothberger if and only if
    \(X\) has a countable topology.
\end{corollary}
\begin{proof}
    Suppose \(X\) has a countable topology and let \(\{ U_n : n \in \omega \}\) be a faithful enumeration
    of \(\mathscr T_X\).
    Note that, for each \(x \in X\), there exists some \(n(x) \in \omega\) such that \(X \setminus \{x\} = U_{n(x)}\).
    This mapping \(x \to n(x)\), \(X \mapsto \omega\), is an injection, so \(X\) is countable.
    Hence, by Proposition \ref{prop:VerySmallMeansLambdaRoth}, \(X\) is Markov \(\lambda\)-Rothberger.

    On the other hand, if \(X\) is Markov \(\lambda\)-Rothberger, Proposition \ref{prop:MarkLambdaRothbergerIsTopCount}
    asserts that \(X\) is topologically countable.
    Since \(X\) is \(T_1\), \(\mathcal N_x = \{x\}\) for each \(x \in X\).
    Hence, \(X\) is countable.
    Thus, by Theorem \ref{thm:CountableNotMarkoveLambdaRoth}, \(X\) must have a countable topology.
\end{proof}
\begin{question}
    Can the \(T_1\) condition in Corollary \ref{cor:MarkovLambdaRothberger} be dropped, or is there
    an example of a non-\(T_1\) topologically countable space with an uncountable topology that is
    Markov \(\lambda\)-Rothberger?
\end{question}
\begin{example}
    The space \(X = \omega\) with the cofinite topology is an infinite \(T_1\) space
    which is Markov \(\lambda\)-Rothberger since it is countable and has a countable topology.
\end{example}
\begin{lemma} \label{lem:MarkovRothbergerThing}
    For any space \(X\),
    \[\mathrm{II} \underset{\mathrm{mark}}{\uparrow} \mathsf G_1(\mathcal O_X, \mathcal O_X)
    \iff \mathrm{II} \underset{\mathrm{mark}}{\uparrow} \mathsf G_1(\Omega_X, \Lambda_X)
    \iff \mathrm{II} \underset{\mathrm{mark}}{\uparrow} \mathsf G_1(\Omega_X, \mathcal O_X).\]
\end{lemma}
\begin{proof}
    Note that the implication
    \[\mathrm{II} \underset{\mathrm{mark}}{\uparrow} \mathsf G_1(\Omega_X, \Lambda_X)
    \implies \mathrm{II} \underset{\mathrm{mark}}{\uparrow} \mathsf G_1(\Omega_X, \mathcal O_X)
    \implies \mathrm{II} \underset{\mathrm{mark}}{\uparrow} \mathsf G_1(\mathcal O_X, \mathcal O_X)\]
    follows from Proposition \ref{prop:MarkLambdaRothbergerIsTopCount} and
    \cite[Theorem 4.17]{CCHMengerRothbergerSurvey}.
    So we need only show that
    \[\mathrm{II} \underset{\mathrm{mark}}{\uparrow} \mathsf G_1(\mathcal O_X, \mathcal O_X)
    \implies \mathrm{II} \underset{\mathrm{mark}}{\uparrow} \mathsf G_1(\Omega_X, \Lambda_X).\]
    Indeed, by \cite[Theorem 4.17]{CCHMengerRothbergerSurvey},
    \[\mathrm{II} \underset{\mathrm{mark}}{\uparrow} \mathsf G_1(\mathcal O_X, \mathcal O_X)
    \iff \mathrm{II} \underset{\mathrm{mark}}{\uparrow} \mathsf G_1(\Omega_X, \Omega_X).\]
    Consequently, as \(\Omega_X \subseteq \Lambda_X\), we see that
    \[\mathrm{II} \underset{\mathrm{mark}}{\uparrow} \mathsf G_1(\mathcal O_X, \mathcal O_X)
    \iff \mathrm{II} \underset{\mathrm{mark}}{\uparrow} \mathsf G_1(\Omega_X, \Omega_X)
    \implies \mathrm{II} \underset{\mathrm{mark}}{\uparrow} \mathsf G_1(\Omega_X, \Lambda_X),\]
    finishing the proof.
\end{proof}
\begin{theorem} \label{thm:SingleLambdaOpen}
    For any space \(X\),
    \[\mathsf G_1(\Lambda_X, \Lambda_X) \leq_{\mathrm{II}}^+ \mathsf G_1(\mathcal O_X, \mathcal O_X)
    \leftrightarrows \mathsf G_1(\Omega_X, \Lambda_X)
    \leftrightarrows \mathsf G_1(\Omega_X, \mathcal O_X).\]
    In particular, the \(\lambda\)-Rothberger and Rothberger games are equivalent for all levels of strategy
    constituting \(\leq_{\mathrm{II}}^+\) except for the Markov strategy level for the second player.
\end{theorem}
\begin{proof}
    We first address
    \[\mathsf G_1(\Lambda_X, \Lambda_X) \leq_{\mathrm{II}}^+ \mathsf G_1(\mathcal O_X, \mathcal O_X).\]
    Note that Lemma \ref{lem:BasicLambdaLindelof} gives the equivalence for the level of
    constant strategies for the first player,
    Lemma \ref{lem:RothbergerEquivalence} gives the equivalences of the predetermined
    and perfect-information strategy levels for the first player, and
    Lemma \ref{lem:TwoSingleStrategic} gives the equivalence of the perfect-information strategies
    for the second player.
    By Proposition \ref{prop:MarkLambdaRothbergerIsTopCount}, we see that
    \[\mathrm{II} \underset{\mathrm{mark}}{\uparrow} \mathsf G_1(\Lambda_X, \Lambda_X)
    \implies \mathrm{II} \underset{\mathrm{mark}}{\uparrow} \mathsf G_1(\mathcal O_X, \mathcal O_X),\]
    but, by Corollary \ref{cor:NontrivialSpacesAreNotLambdaRoth}, we see that the converse fails
    since any countably infinite Hausdorff space is Markov Rothberger but not Markov \(\lambda\)-Rothberger.

    For the equivalence
    \[\mathsf G_1(\mathcal O_X, \mathcal O_X) \leftrightarrows \mathsf G_1(\Omega_X, \Lambda_X)
    \leftrightarrows \mathsf G_1(\Omega_X, \mathcal O_X),\]
    combine Lemmas \ref{lem:BasicLambdaLindelof}, \ref{lem:RothbergerEquivalence}, \ref{lem:TwoSingleStrategic},
    and \ref{lem:MarkovRothbergerThing}.
\end{proof}

\section{An Equivalence of Games} \label{section:TheMainTheorem}

We are now poised to prove Theorem \ref{thm:MainTheorem}, which we do via
Lemmas \ref{lem:LambdaFOBelow} and \ref{lem:sCDBelow}.
We note that the combinatorial result, Theorem \ref{thm:SingleLambdaOpen},
greatly simplifies the proof of Lemma \ref{lem:LambdaFOBelow} by allowing us to apply
our Translation Theorem \cite[Theorem 12]{CHContinuousFunctions}, recorded here
as Theorem \ref{thm:TranslationTheorem}.

\begin{theorem}[{\cite[Theorem 12]{CHContinuousFunctions}}]\label{thm:TranslationTheorem}
    Let \(\mathcal A\), \(\mathcal B\), \(\mathcal C\), and \(\mathcal D\) be collections.
    Suppose there are functions \(\overleftarrow{\mathrm T}_{\mathrm{I},n}:\mathcal B \to \mathcal A\) and
    \(\overrightarrow{\mathrm T}_{\mathrm{II},n}: \left( \bigcup \mathcal A \right) \times \mathcal B \to \bigcup \mathcal B\)
    for each \(n \in \omega\) such that,
    \begin{enumerate}[label=(T\Roman*)]
        \item \label{TranslationFirst}
        if \(x \in \overleftarrow{\mathrm T}_{\mathrm{I},n}(B)\), then \(\overrightarrow{\mathrm T}_{\mathrm{II},n}(x,B) \in B\); and
        \item \label{TranslationSecond}
        if \(\langle x_n : n \in \omega \rangle \in \prod_{n\in\omega} \overleftarrow{\mathrm T}_{\mathrm{I},n}(B_n)\)
        and \(\{x_n : n\in\omega\} \in \mathcal C\),
        then \[\left\{\overrightarrow{\mathrm T}_{\mathrm{II},n}(x_n,B_n) : n \in \omega \right\} \in \mathcal D.\]
    \end{enumerate}
    Then \(\mathsf G_1(\mathcal A,\mathcal C) \leq_{\mathrm{II}} \mathsf G_1(\mathcal B, \mathcal D)\).
\end{theorem}
\begin{lemma} \label{lem:LambdaFOBelow}
    Let \(X\) be an infinite Tychonoff space.
    Then
    \[\mathsf G_1(\mathcal O_X, \mathcal O_X)
    \dualarrows \mathsf G_1(\mathcal N[X_{\mathrm{fin}}], \neg \mathcal O_X)
    \leftrightarrows \mathsf G_1(\mathcal N[X_{\mathrm{fin}}], \neg \Lambda_X)
    \leq_{\mathrm{II}} \mathsf G_1(\mathscr T_{C_p(X)}, s\mathrm{CD}_{C_p(X)}).\]
\end{lemma}
\begin{proof}
    By Theorem \ref{thm:SingleLambdaOpen},
    \(\mathsf G_1(\mathcal O_X, \mathcal O_X)
    \leftrightarrows \mathsf G_1(\Omega_X, \Lambda_X) \leftrightarrows \mathsf G_1(\Omega_X, \mathcal O_X)\).
    Moreover, by duality,
    \[\mathsf G_1(\mathcal N[X_{\mathrm{fin}}], \neg \Lambda_X)
    \dualarrows \mathsf G_1(\Omega_X, \Lambda_X)
    \leftrightarrows \mathsf G_1(\Omega_X, \mathcal O_X)
    \dualarrows \mathsf G_1(\mathcal N[X_{\mathrm{fin}}], \neg \mathcal O_X).\]
    Hence,
    \[\mathsf G_1(\mathcal N[X_{\mathrm{fin}}], \neg \Lambda_X)
    \leftrightarrows \mathsf G_1(\mathcal N[X_{\mathrm{fin}}], \neg \mathcal O_X).\]
    
    We now prove that
    \[\mathsf G_1(\mathcal N[X_{\mathrm{fin}}], \neg \mathcal O_X)
    \leq_{\mathrm{II}} \mathsf G_1(\mathscr T_{C_p(X)}, s\mathrm{CD}_{C_p(X)})\]
    via an application of Theorem \ref{thm:TranslationTheorem}.
    Fix a choice \[( f_W, F_W , \varepsilon_W ) \in C_p(X) \times X_{\mathrm{fin}} \times (0,1)\]
    for each \(W \in \mathscr T_{C_p(X)}\) such that
    \([f_W ; F_W , \varepsilon_W ] \subseteq W\).
    Then define \(\overleftarrow{\mathrm T}_{\mathrm{I},n} : \mathscr T_{C_p(X)} \to \mathcal N[X_{\mathrm{fin}}]\)
    by \(\overleftarrow{\mathrm T}_{\mathrm{I},n}(W) = \mathcal N_{F_W}\).
    Define, also, \(\overrightarrow{\mathrm T}_{\mathrm{II},n} : \mathscr T_X \times \mathscr T_{C_p(X)} \to C_p(X)\)
    to be such that
    \[\overrightarrow{\mathrm T}_{\mathrm{II},n}(U, W) \restriction_{F_W} = f_W \restriction_{F_W} \wedge
    \overrightarrow{\mathrm T}_{\mathrm{II},n}(U, W) \restriction_{X \setminus U} = n,\]
    if such a function exists.
    If no such function exists, we can define \(\overrightarrow{\mathrm T}_{\mathrm{II},n}(U,W)\)
    arbitrarily.

    Observe that, if \(W \in \mathscr T_{C_p(X)}\) and \(U \in \overleftarrow{\mathrm T}_{\mathrm{I},n}(W)\),
    then \(F_W \subseteq U\) and \(\overrightarrow{\mathrm T}_{\mathrm{II},n}(U, W)\) consequently has the property
    that \[\overrightarrow{\mathrm T}_{\mathrm{II},n}(U, W) \restriction_{F_W} = f_W \restriction_{F_W}.\]
    Hence,
    \[\overrightarrow{\mathrm T}_{\mathrm{II},n}(U, W) \in [f_W ; F_W, \varepsilon_W ] \subseteq W.\]
    So \ref{TranslationFirst} is satisfied.

    To establish \ref{TranslationSecond}, suppose \(\langle W_n : n \in \omega \rangle \in \mathscr T_{C_p(X)}^\omega\)
    and \[\langle U_n : n \in \omega \rangle \in \prod_{n\in\omega} \overleftarrow{\mathrm T}_{\mathrm{I},n}(W_n)\]
    are such that \(\{ U_n : n \in \omega \} \not\in \mathcal O_X\).
    Note then that we can choose \(\bar{x} \in X \setminus \bigcup \{ U_n : n \in \omega \}\).
    For brevity, let \(g_n = \overrightarrow{\mathrm T}_{\mathrm{II},n}(U_n, W_n)\) for each \(n \in \omega\).
    We show that \(\{ g_n : n \in \omega \} \in s\mathrm{CD}_{C_p(X)}\).
    Indeed, suppose \(g \in C_p(X)\) and consider \(\varepsilon = 1\).
    Then, let \(M \in \omega\) be such that \(g(\bar{x}) + \varepsilon < M\).
    Since \(g_n(\bar{x}) = n\) for each \(n \in \omega\), we see that
    \(| \{ g_n : n \in \omega \} \cap [ g; \{\bar{x}\}, \varepsilon ] | \leq M\),
    finishing the proof.
\end{proof}

We note that, when the second player can keep track of their play history,
then they can guarantee that they form an injective sequence of functions
which then guarantees the conclusion of \cite[Lemma 3.4(a)]{Chiozini}.
We make this explicit by choosing the function space neighborhoods in a way
that necessitates pairwise distinctness of the chosen functions in the proof below in both
perfect-information implications and the predetermined level for P1.
In the Markov level for the second player, however, we must exercise a bit more care to guarantee
that the sequence of functions chosen in the game on \(C_p(X)\) is pairwise distinct.

In the proof, we will use the bijection \(F \mapsto \mathcal N_F\),
\(X_{\mathrm{fin}} \to \mathcal N[X_{\mathrm{fin}}]\), to simplify some notation.
\begin{lemma} \label{lem:sCDBelow}
     Let \(X\) be an infinite Tychonoff space.
     Then 
     \[
        \mathsf G_1(\mathscr T_{C_p(X)}, s\mathrm{CD}_{C_p(X)})
        \leq_{\mathrm{II}} \mathsf G_1(\mathcal N[X_{\mathrm{fin}}], \neg \Lambda_X).
     \]
\end{lemma}
\begin{proof}
    Fix a function \(\gamma : X_{\mathrm{fin}} \to X\) such that
    \(\gamma(F) \in X \setminus F\), which is possible since \(X\) is infinite.
    We will use this function throughout the proof to ensure that the open sets we define are proper.

    We first address the implication
    \[\mathrm{II} \underset{\mathrm{mark}}{\uparrow} \mathsf G_1(\mathscr T_{C_p(X)}, s\mathrm{CD}_{C_p(X)})
    \implies \mathrm{II} \underset{\mathrm{mark}}{\uparrow} \mathsf G_1(\mathcal N[X_{\mathrm{fin}}], \neg \Lambda_X).\]
    By duality, we have that
    \begin{align*}
        \mathrm{II} \underset{\mathrm{mark}}{\uparrow} \mathsf G_1(\mathcal N[X_{\mathrm{fin}}], \neg \Lambda_X)
        \iff \mathrm{I} \underset{\mathrm{pre}}{\uparrow} \mathsf G_1(\Omega_X, \Lambda_X)
        &\iff X \not\models \mathsf S_1(\Omega, \Lambda),
    \end{align*}
    so, we will show that
    \[\mathrm{II} \underset{\mathrm{mark}}{\uparrow} \mathsf G_1(\mathscr T_{C_p(X)}, s\mathrm{CD}_{C_p(X)})
    \implies X \not\models \mathsf S_1(\Omega, \Lambda)\]
    by way of the contrapositive.
    So suppose \(X \models \mathsf S_1(\Omega, \Lambda)\) and suppose
    \(\tau : \mathscr T_{C_p(X)} \times \omega \to C_p(X)\) is given with the property
    that \(\tau(W,n) \in W\) for each \(W \in \mathscr T_{C_p(X)}\), \(n \in \omega\);
    i.e. \(\tau\) is a Markov strategy.
    We will use \(\tau\) to define an associated strategy \(\sigma\) for P1 in \(\mathsf G_1(\Omega_X, \Lambda_X)\)
    as follows.
    We will then use the fact that \(\sigma\) cannot be winning to show that \(\tau\) is not winning.

    First, for each \(F \in X_{\mathrm{fin}}\) and \(n \in \omega\), let \(h_{F} \in C_p(X)\)
    be such that \(h_{F}[F] = \{0\}\) and \(h_{F}(\gamma(F)) = 3\).
    We will now recursively define structures for \(F \in X_{\mathrm{fin}}\),
    \(F^\ast := F \cup \{ \gamma(F) \}\),
    and \(\langle F_j : j < n \rangle \in X_{\mathrm{fin}}^n\).
    Suppose that, for \(n \in \omega\), we have defined
    \(W_{G, \langle F_\ell : \ell < j \rangle}\) for each \(j < n \) and \(G \in X_{\mathrm{fin}}\).
    Then let
    \[W_{F,\langle F_j : j < n \rangle}
    = [ h_{F}; F^\ast, 2^{-n}] \setminus \{ \tau(W_{F_j, \langle F_\ell: \ell < j \rangle}, j) : j < n \}\]
    and note that \(W_{F, \langle F_j : j < n \rangle} \in \mathscr T_{C_p(X)}\).
    Consider now
    \[U_{F,\langle F_j : j < n \rangle} = \tau(W_{F,\langle F_j : j < n \rangle}, n)^{-1}\left(-2^{-n}, 2^{-n} \right).\]
    We claim that \(U_{F,\langle F_j : j < n \rangle} \in \mathscr T_X\).
    First, note that \(F \subseteq U_{F,\langle F_j : j < n \rangle}\) since, for \(x \in F\),
    \[\tau(W_{F,\langle F_j : j < n \rangle}, n) \in W_{F,\langle F_j : j < n \rangle} \subseteq [ h_{F}; F^\ast, 2^{-n}]
    \implies \left| \tau(W_{F,\langle F_j : j < n \rangle}, n)(x) \right| < 2^{-n}.\]
    Then, for \(x = \gamma(F)\), note that
    \begin{align*}
        &\left| \tau(W_{F,\langle F_j : j < n \rangle}, n)(x) - h_F(x) \right| < 2^{-n}\\
        &\implies \tau(W_{F,\langle F_j : j < n \rangle}, n)(x) > h_{F}(x) - 2^{-n} = 3 - 2^{-n} > 2^{-n}.
    \end{align*}
    That is, \(\gamma(F) = x \not\in U_{F,\langle F_j : j < n \rangle}\), so \(U_{F,\langle F_j : j < n \rangle} \in \mathscr T_X\).

    Now, for \(\langle F_j : j < n \rangle \in X_{\mathrm{fin}}^n\), \(n\in\omega\), let
    \[\mathscr U_{\langle F_j : j < n \rangle} = \{ U_{F, \langle F_j : j < n \rangle} : F \in X_{\mathrm{fin}} \} \in \Omega_X\]
    and fix a choice \(\mathbf F_{{\langle F_j : j < n \rangle}} : \mathscr U_{\langle F_j : j < n \rangle} \to X_{\mathrm{fin}}\)
    to be such that, for each \(U \in \mathscr U_{\langle F_j : j < n \rangle}\),
    \[U = U_{\mathbf F_{{\langle F_j : j < n \rangle}}(U), \langle F_j : j < n \rangle}\]
    and, for \(U_1, U_2 \in \mathscr U_{\langle F_j : j < n \rangle}\),
    \[U_1 := U_{F, {\langle F_j : j < n \rangle}} = U_{G, {\langle F_j : j < n \rangle}} =: U_2
    \implies \mathbf F_{\langle F_j : j < n \rangle}(U_1) = \mathbf F_{\langle F_j : j < n \rangle}(U_2).\]

    Initialize
    \[\sigma\langle \rangle = \mathscr U_{\langle \rangle }\]
    and suppose, for \(n \in \omega\), we have a sequence \(\langle U_j : j \leq n \rangle\)
    of legal plays by P2 and a sequence \(\langle F_j : j \leq n \rangle\)
    where
    \[F_j = \mathbf F_{\langle F_\ell : \ell < j \rangle}(U_j)\]
    for each \(j \leq n\).
    Then
    \[\sigma\langle U_j : j \leq n \rangle = \mathscr U_{\langle F_j : j \leq n \rangle}.\]
    Now, since \(X \models \mathsf S_1(\Omega, \Lambda)\), by Lemma \ref{lem:RothbergerEquivalence},
    \(\sigma\) cannot be winning.
    So there must exist some sequence \(\langle U_n : n \in \omega \rangle\) of legal plays
    by P2 such that \(\{ U_n : n \in \omega \} \in \Lambda_X\).
    Let \(\langle F_n : n \in \omega \rangle\) be the associated sequence of finite subsets of \(X\)
    such that \(F_n = \mathbf F_{\langle F_j : j < n \rangle}(U_n)\) for each \(n \in \omega\).

    Let \(g_n = \tau(W_{F_n, \langle F_j : j < n \rangle}, n)\) for \(n \in \omega\).
    Note that, since
    \begin{align*}
        g_n = \tau(W_{F_n, \langle F_j : j < n \rangle}, n)
        \in W_{F_n, \langle F_j : j < n \rangle}
        &= [ h_{F_n}; F_n^\ast, 2^{-n}] \setminus \{ \tau(W_{F_j, \langle F_\ell: \ell < j \rangle}, j) : j < n \}\\
        &= [ h_{F_n}; F_n^\ast, 2^{-n}] \setminus \{ g_j : j < n \},
    \end{align*}
    \(\{ g_n : n \in \omega \}\) is an infinite family of pairwise distinct functions.
    
    We claim that \(\{ g_n : n \in \omega \} \not\in s\mathrm{CD}_{C_p(X)}\).
    Consider \(\mathbf 0\) and let \(x \in X\) and \(\varepsilon > 0\) be arbitrary.
    Then, since \(\{ U_n : n \in \omega \} \in \Lambda_X\), we can let \(M \in \omega\)
    be such that \(2^{-M} < \varepsilon\) and \(x \in U_M\).
    Moreover, note that \(A := \{ n \in \omega : n \geq M \wedge x \in U_n \}\)
    is an infinite set.
    Now consider any \(n \in A\) and observe that
    \begin{align*}
    &x \in U_n = U_{F_n, \langle F_j : j < n \rangle}
    = \tau(W_{F_n,\langle F_j : j < n \rangle}, n)^{-1}\left(-2^{-n}, 2^{-n} \right)
    = g_n^{-1}\left(-2^{-n}, 2^{-n}\right)\\
    &\implies |g_n(x)| < 2^{-n} \leq 2^{-M} < \varepsilon\\
    &\implies g_n \in [\mathbf 0; \{x\} , \varepsilon].
    \end{align*}
    Since the \(\{g_n:n\in\omega\}\) are pairwise distinct and, for every \(n \in A\),
    \(g_n \in [\mathbf 0; \{x\}, \varepsilon]\), we see that
    \([\mathbf 0; \{x\}, \varepsilon] \cap \{ g_n : n \in \omega \}\) is infinite.
    That is, \(\{ g_n : n \in \omega \} \not\in s\mathrm{CD}_{C_p(X)}\).

    Since \(\tau\) was arbitrary, we see that P2 does not have a winning Markov
    strategy in \(\mathsf G_1(\mathscr T_{C_p(X)}, s\mathrm{CD}_{C_p(X)})\).

    For the remaining implications, we will define, for \(g \in C_p(X)\), \(F \in X_{\mathrm{fin}}\), and \(n \in \omega\),
    \[B_{F,n} = \left[\mathbf{2^{-n}}; F, 2^{-n-2}\right]\]
    and \[W_{g,F,n} = g^{-1}\left(3 \cdot 2^{-n-2}, 5 \cdot 2^{-n-2}\right) \setminus \{\gamma(F)\}.\]

    First, we note that, if \(g \in B_{F,n}\), then \(F \subseteq W_{g,F,n} \in \mathscr T_X\).
    The fact that \(W_{g,F,n}\) is a proper open subset of \(X\) follows immediately from the continuity
    of \(g\) and that \(\gamma(F)\) is deleted from it.
    To see that \(F \subseteq W_{g,F,n}\), note that, for \(x \in F\),
    \[-2^{-n-2} < g(x) - 2^{-n} < 2^{-n-2} \implies 3 \cdot 2^{-n-2} < g(x) < 5 \cdot 2^{-n-2}.\]
    That is, \(x \in g^{-1}\left(3 \cdot 2^{-n-2}, 5 \cdot 2^{-n-2}\right)\).
    Since \(\gamma(F) \not\in F\), we see that \(F \subseteq W_{g,F,n}\).

    We also note that, if \(F \subseteq G\), \(F,G \in X_{\mathrm{fin}}\), and \(m < n\),
    then \(B_{F,m} \cap B_{G,n} = \varnothing\).
    Indeed, suppose \(g \in B_{F,m}\) and \(h \in B_{G,n}\).
    By the above computation, we observe that, for \(x \in F \subseteq G\),
    \[h(x) < 5 \cdot 2^{-n-2} < 6 \cdot 2^{-n-2} = 3 \cdot 2^{-n-1} \leq 3 \cdot 2^{-m-2} < g(x).\]
    In particular, \(g \neq h\), so \(B_{F,m} \cap B_{G,n} = \varnothing\).
    We will use these observations in the remainder of this proof.

    We prove that
    \[\mathrm{II} \uparrow \mathsf G_1(\mathscr T_{C_p(X)}, s\mathrm{CD}_{C_p(X)})
    \implies \mathrm{II} \uparrow \mathsf G_1(\mathcal N[X_{\mathrm{fin}}], \neg \Lambda_X).\]
    So suppose \(\hat{\tau}\) is a winning strategy for P2 in \(\mathsf G_1(\mathscr T_{C_p(X)}, s\mathrm{CD}_{C_p(X)})\).
    We define a strategy \(\tau\) for P2 in \(\mathsf G_1(\mathcal N[X_{\mathrm{fin}}], \neg \Lambda_X)\)
    as follows.
    For any \(\langle F_j : j \leq n \rangle \in X_{\mathrm{fin}}^{n+1}\), let, for each \(j \leq n\),
    \(G_j = \bigcup \{ F_\ell : \ell \leq j \} \in X_{\mathrm{fin}}\).
    Then, we can let \(g = \hat\tau\langle B_{G_j, j} : j \leq n \rangle \in B_{G_n,n}\) and define
    \[\tau\langle F_j : j \leq n \rangle = W_{g,G_n, n}.\]
    Since \(F_n \subseteq G_n \subseteq W_{g,G_n, n}\), this is a legal response.
    Hence, \(\tau\) is defined.

    We now argue that \(\tau\) is winning.
    So let \(\langle F_n : n \in \omega \rangle \in X_{\mathrm{fin}}^\omega\) be given and set,
    for each \(n \in \omega\), \(G_n = \bigcup \{ F_j : j \leq n \}\).
    Let, also, \(g_n = \hat\tau\langle B_{G_j, j} : j \leq n \rangle\) for each \(n \in \omega\).
    We note that, by the comment above, for \(m < n\), since \(G_m \subseteq G_n\), \(g_m \neq g_n\).
    That is, \(\{ g_n : n \in \omega \}\) consists of pairwise distinct functions.
    In particular, \(\{ g_n : n \in \omega \}\) is an infinite set.
    Moreover, since \(\hat\tau\) is winning, \(\{ g_n : n \in \omega \} \in s\mathrm{CD}_{C_p(X)}\).
    We claim that \[\{ \tau\langle F_j : j \leq n \rangle : n \in \omega \} = \{ W_{g_n,G_n,n} : n \in \omega \} \not\in \Lambda_X.\]
    Indeed, since \(\{ g_n : n \in \omega \} \in s\mathrm{CD}_{C_p(X)}\), we can find \(x \in X\) and \(\varepsilon > 0\)
    such that \([\mathbf 0; \{x\} , \varepsilon ] \cap \{ g_n : n \in \omega \}\) is finite.
    Hence, we can let \(M \in \omega\) be such that, for each \(n \geq M\),
    \(5 \cdot 2^{-n-2} \leq \varepsilon\) and \(g_n \not\in [\mathbf 0; \{x\} , \varepsilon]\).
    For each \(n \geq M\), we have that either \(g_n(x) \leq -\varepsilon < 0 < 3 \cdot 2^{-n-2}\) or
    \[g_n(x) \geq \varepsilon \geq 5 \cdot 2^{-n-2}.\]
    Consequently,
    \[x \not\in W_{g_n, G_n, n} = \tau \langle F_j : j \leq n \rangle\]
    for each \(n \geq M\).
    Hence, \[\{ \tau\langle F_j : j \leq n \rangle : n \in \omega \} = \{ W_{g_n,G_n,n} : n \in \omega \} \not\in \Lambda_X,\]
    and P2 has won.

    We now prove that
    \[\mathrm{I} \not\uparrow \mathsf G_1(\mathscr T_{C_p(X)}, s\mathrm{CD}_{C_p(X)})
    \implies \mathrm{I} \not\uparrow \mathsf G_1(\mathcal N[X_{\mathrm{fin}}], \neg \Lambda_X).\]
    So let \(\sigma\) be any strategy for P1 in \(\mathsf G_1(\mathcal N[X_{\mathrm{fin}}], \neg \Lambda_X)\).
    We will define a strategy \(\hat\sigma\) for P1 in \(\mathsf G_1(\mathscr T_{C_p(X)}, s\mathrm{CD}_{C_p(X)})\)
    as follows.
    Let \(G_0 = F_0 = \sigma\langle \rangle\) and set \(\hat\sigma\langle \rangle = B_{G_0,0}\).
    Note that, for \(g_0 \in \hat\sigma\langle \rangle\),
    \(F_0 \subseteq G_0 \subseteq W_{g_0,G_0,0} \in \mathscr T_X\).
    Now, for \(n \in \omega\), suppose we have defined \(\langle F_j : j \leq n \rangle\),
    \(\langle G_j : j \leq n \rangle\), \(\langle g_j : j \leq n \rangle\), and
    \(\langle W_{g_j, G_j, j} : j \leq n \rangle\) as above.
    We then consider \(F_{n+1} = \sigma \langle W_{g_j, G_j, j} : j \leq n \rangle\) and set
    \(G_{n+1} = G_n \cup F_{n+1}\).
    We then define
    \[\hat\sigma \langle g_j : j \leq n \rangle = B_{G_{n+1},n+1}.\]
    Note that, for \(g_{n+1} \in \hat\sigma \langle g_j : j \leq n \rangle = B_{G_{n+1},n+1}\),
    \(F_{n+1} \subseteq G_{n+1} \subseteq W_{g_{n+1}, G_{n+1}, n+1}\).
    This completes our definition of \(\hat\sigma\).

    Since \(\hat\sigma\) cannot be winning, there must exist some
    \(\langle F_n : n \in \omega \rangle\), \(\langle G_n : n \in \omega \rangle\),
    and \(\langle g_n : n \in \omega \rangle\) satisfying the relations above
    such that \(\{ g_n : n \in \omega \} \in s\mathrm{CD}_{C_p(X)}\).
    We claim that \(\{ W_{g_n, G_n, n} : n \in \omega \} \not\in \Lambda_X\) and, thus,
    is a legal run of the game by P2 against \(\sigma\) where P2 wins.
    Note that, by our definitions, \(\{ g_n : n \in \omega \}\) consists of
    pairwise distinct functions.
    Then, by the exact argument in the case for perfect-information strategies
    for P2 above, we see that \(\{ W_{g_n, G_n, n} : n \in \omega \} \not\in \Lambda_X\).

    Finally, we prove that
    \[\mathrm{I} \underset{\mathrm{pre}}{\not\uparrow} \mathsf G_1(\mathscr T_{C_p(X)}, s\mathrm{CD}_{C_p(X)})
    \implies \mathrm{I} \underset{\mathrm{pre}}{\not\uparrow} \mathsf G_1(\mathcal N[X_{\mathrm{fin}}], \neg \Lambda_X)\]
    via the contrapositive.
    So, suppose \(\mathrm{I} \underset{\mathrm{pre}}{\uparrow} \mathsf G_1(\mathcal N[X_{\mathrm{fin}}], \neg \Lambda_X)\)
    and note that, by duality and Theorem \ref{thm:SingleLambdaOpen},
    \begin{align*}
        \mathrm{I} \underset{\mathrm{pre}}{\uparrow} \mathsf G_1(\mathcal N[X_{\mathrm{fin}}], \neg \Lambda_X)
        &\iff \mathrm{II} \underset{\mathrm{mark}}{\uparrow} \mathsf G_1(\Omega_X, \Lambda_X)\\
        &\iff \mathrm{II} \underset{\mathrm{mark}}{\uparrow} \mathsf G_1(\mathcal O_X, \mathcal O_X).
    \end{align*}
    Then, \(X\) is countable by \cite[Corollary 4.18]{CCHMengerRothbergerSurvey}.
    So let \(X = \{ x_n : n \in \omega \}\) and \(F_n = \{ x_j : j \leq n \}\) for \(n \in \omega\).
    We claim that \(\langle B_{F_n,n} : n \in \omega \rangle\) is a predetermined
    winning strategy for P1 in \(\mathsf G_1(\mathscr T_{C_p(X)}, s\mathrm{CD}_{C_p(X)})\).
    So suppose \(g_n \in B_{F_n,n}\) for each \(n \in \omega\).
    Since the \(\{B_{F_n,n} : n \in \omega\}\) is a collection of pairwise disjoint open sets,
    \(\{ g_n : n \in \omega \}\) consists of pairwise distinct functions.
    From this, it is clear that, for any \(x \in X\) and \(\varepsilon > 0\),
    \([\mathbf 0; \{x\}, \varepsilon] \cap \{ g_n : n \in \omega \}\) is infinite, and,
    thus, that \(\{g_n:n\in\omega\} \not\in s\mathrm{CD}_{C_p(X)}\).
    That is, P1 has a predetermined winning strategy in \(\mathsf G_1(\mathscr T_{C_p(X)}, s\mathrm{CD}_{C_p(X)})\).
\end{proof}
By combining Lemma \ref{lem:LambdaFOBelow} and \ref{lem:sCDBelow}, we obtain:
\begin{theorem} \label{thm:MainTheorem}
    For any infinite Tychonoff space \(X\),
    \[\mathsf G_1(\mathcal O_X, \mathcal O_X)
    \dualarrows
    \mathsf G_1(\mathcal N[X_{\mathrm{fin}}], \neg \mathcal O_X)
    \leftrightarrows
    \mathsf G_1(\mathcal N[X_{\mathrm{fin}}], \neg \Lambda_X)
    \equiv \mathsf G_1(\mathscr T_{C_p(X)}, s\mathrm{CD}_{C_p(X)}).\]
    In particular, the Rothberger game on \(X\) is perfect- and Markov-information
    dual to the strong closed discrete game on \(C_p(X)\),
    for infinite Tychonoff \(X\).
\end{theorem}

\section*{Acknowledgments}

We would first like to thank Lucas Chiozini for sharing his paper \cite{Chiozini} with us and inspiring this work
to be done.
We would also like to thank David Guerrero S\'{a}nchez for drawing our attention to \cite{GuerreroSanchezTkachuk2017}
while at the 40th Summer Conference on Topology and Its Applications at the University of Split, Croatia.
The first named author is grateful to the National Science Foundation for supporting his travel to the above-referenced
conference under the DMS Grant Number 2555991.

Finally, we would like to acknowledge the use of ChatGPT 5.4 in assisting with general literature review and ChatGPT 5.6
in helping suggest a proof sketch of an initial draft of the proof of Theorem \ref{thm:CountableNotMarkoveLambdaRoth} in the particular
case for the discrete countable space, and for suggesting improvements to the proof of Lemma \ref{lem:sCDBelow}.
We did not, however, use any AI in the writing of this manuscript.

\providecommand{\bysame}{\leavevmode\hbox to3em{\hrulefill}\thinspace}
\providecommand{\MR}{\relax\ifhmode\unskip\space\fi MR }
\providecommand{\MRhref}[2]{%
  \href{http://www.ams.org/mathscinet-getitem?mr=#1}{#2}
}
\providecommand{\href}[2]{#2}

\end{document}